\documentclass[bibliography=totocnumbered,final
]{scrartcl}
\usepackage[utf8]{inputenc}
\usepackage{hyperref}
\usepackage[english]{babel}
\usepackage{enumerate}
\usepackage{aliascnt}
\usepackage{amssymb}
\usepackage{amsmath}
\usepackage{verbatim}
\usepackage{mathtools}
\usepackage{nicefrac}
\usepackage{tikz}
\usepackage{subfig}
\usepackage[nottoc]{tocbibind}

\usepackage{amsthm}

\usepackage[title]{appendix}

\usepackage[notcite,notref]{showkeys}
\usepackage{color}

\usepackage{dsfont} %f\"ur Indikatorfunktion \mathds{1}
\usepackage{hyperref} %Links, autoref
\usepackage{cleveref}

\usepackage{autonum} % Load last!

\allowdisplaybreaks[4] %0-4

\newtheorem{prop}{Proposition}[section]

\newaliascnt{lem}{prop} 

\newtheorem{lem}[lem]{Lemma}%[section]

\aliascntresetthe{lem} 
\Crefname{lem}{Lemma}{Lemmas}

\newaliascnt{defi}{prop} 
 \newtheorem{defi}[defi]{Definition}%[section

\aliascntresetthe{defi} 
\Crefname{defi}{Definition}{Definitions}

\newaliascnt{cor}{prop} 
\aliascntresetthe{cor}

\newaliascnt{remark}{prop} 
 \newtheorem{remark}[remark]{Remark}%[section]
 
\aliascntresetthe{remark} 

\newaliascnt{thm}{prop} 
 \newtheorem{thm}[thm]{Theorem}%[section]
 
\aliascntresetthe{thm} 

\newaliascnt{example}{prop} 
\aliascntresetthe{example}

\def\equationautorefname~#1\null{%
  (#1)\null
}

\DeclareMathOperator{\supp}{supp} 
\newcommand{\R}{\ensuremath{\mathbb{R}}}
\newcommand{\N}{\ensuremath{\mathbb{N}}}
\newcommand{\Z}{\ensuremath{\mathbb{Z}}}
\renewcommand{\S}{\ensuremath{\mathbb{S}}}

\newcommand*\diff{\mathop{}\!\mathrm{d}}

\newcommand{\defeq}{\vcentcolon=}

\newcommand{\CalE}{\ensuremath{\mathcal{E}}}

\newcommand{\CalL}{\ensuremath{\mathcal{L}}}
\newcommand{\CalA}{\ensuremath{\mathcal{A}}}

\newcommand{\CalK}{\ensuremath{\mathcal{K}}}
\newcommand{\am}{\mathrm{am} }
\newcommand{\cn}{\mathrm{cn}}
\newcommand{\dn}{\mathrm{dn}}
\newcommand{\sn}{\mathrm{sn}}
\newcommand{\ds}{\;\mathrm{d}s}

\newcommand{\sech}{\mathrm{sech}}

\newcommand{\vKap}{\ensuremath{\vec{\kappa}}}
\newcommand{\abs}[1]{\ensuremath{\left \lvert #1\right\rvert}}
\newcommand{\Norm}[2]{\ensuremath{\left\Vert #1 \right\Vert_{#2}}}
\newcommand{\norm}[2]{\ensuremath{\Vert #1 \Vert_{#2}}}

\title{A Li--Yau inequality for the 1-dimensional Willmore energy}
\author{Marius Müller\thanks{Mathematisches Institut,  Albert--Ludwigs--Universität Freiburg, 79104 Freiburg im Breisgau, Germany.}~and Fabian Rupp\thanks{Institute of Applied Analysis, Ulm University, Helmholtzstra\ss e 18, 89081 Ulm, Germany.}}

\begin{document}

\date{}

\maketitle

\begin{abstract}
\noindent \textbf{Abstract:} By the classical Li--Yau inequality, an immersion of a {closed} surface in $\R^n$ with Willmore energy below $8\pi$ has to be embedded. We discuss analogous results for curves in $\R^2$, involving Euler's elastic energy and other possible curvature functionals. Additionally, we provide applications to associated gradient flows.
\end{abstract}

%\tableofcontents

 

\section{Introduction and main results}
	For an immersion $f\colon\Sigma \to\R^n$ of a surface $\Sigma$, its \emph{Willmore energy} is defined by
	\begin{align}\label{eq:DefWillmore}
		\mathcal{W}(f) \defeq \frac{1}{4}\int_{\Sigma} \abs{H}^2 \diff \mu.
	\end{align}
	Here $H$ denotes the mean curvature vector and
	$\mu$ is the Riemannian measure induced by pulling back the Euclidean metric to $\Sigma$. In their fundamental work \cite{LY82}, Li and Yau proved an inequality which yields that an immersion with suitably small Willmore energy must in fact be an embedding. More specifically, if $\Sigma$ is compact, then 
	\begin{align}\label{eq:LYWillmore}
	\mathcal{W}(f) <8\pi \text{ implies that } f \text{ is an embedding.}
	\end{align}
	Moreover, as a doubly covered round sphere shows, the constant $8\pi$ in \eqref{eq:LYWillmore} is optimal.
	
	In this article, we study the question whether an analogous result as in \eqref{eq:LYWillmore} is true for planar curves. For a closed smooth curve $\gamma\colon\S^1\to \R^2$, which is {immersed}, i.e. $\abs{\gamma^{\prime}}>0$, and has {signed curvature} $\kappa$, its \emph{elastic energy} is defined by
	\begin{align}\label{eq:defElastic}
		\CalE(\gamma) \defeq \int_{\S^1} \abs{\kappa}^2\diff s.
	\end{align}
	This formally resembles the Willmore energy. However, in contrast to \eqref{eq:DefWillmore}, $\CalE$ is not scaling invariant, whereas the property of being embedded is. A natural scaling invariant one-dimensional version of the Willmore energy is the \emph{total curvature}, defined by
	\begin{align}
		\CalK(\gamma)\defeq \int_{\S^1}\abs{\kappa}\diff s.
	\end{align}
	It has a wide range of geometric applications and has been studied {extensively}%excessively
	, for instance in \cite{Fenchel,Fary,MilnorKnots,Milnor2}.
	
	%In this article, 
	We will show that the total curvature does not allow for a non-trivial {version} of \eqref{eq:LYWillmore}. This will be a consequence of the following observation.
	\begin{thm}\label{thm:infTotalCurv}
		We have 
		\begin{align}
			2\pi &= \inf\{ \CalK(\gamma)\mid \gamma\in C^{2}(\S^1;\R^2)  \text{ non-embedded immersion}\} \label{eq:InfTotalCurv} \\
			& = \inf \{ \CalK(\gamma)\mid \gamma\in C^{2}(\S^1;\R^2)  \text{ immersion}\}.
		\end{align}
		Moreover, the infimum among non-embedded immersions is not attained.
	\end{thm}
	{In order to obtain a non-trivial version of \eqref{eq:LYWillmore} we need to identify a different quantity.}
	Our main result shows that the elastic energy provides a positive answer, when restricted to curves of fixed length, or --- equivalently  --- multiplied with the length functional $\CalL$.
	\begin{thm}[Main theorem]\label{thm:main}
	If $\gamma \in W^{2,2}(\mathbb{S}^{1}; \R^{2})$ is an immersed curve with
	\begin{align}
		\CalE(\gamma)\CalL(\gamma)< c^*\defeq   \CalE(\gamma^{\ast})\CalL(\gamma^{\ast}),
		\end{align} 
		then $\gamma$ is an embedding. Here $\gamma^{\ast}$ is the figure eight elastica (see \Cref{def:figeight}).
	\end{thm}
	\begin{remark}
		 The value of $c^*$ is sharp since $\gamma^*$ itself is not an embedding, cf. \Cref{lem:figureeight} for the details. {A numerical computation yields $c^*\simeq 112.4396$}.
	\end{remark}
	We have thus identified a geometric quantity of curves whose smallness ensures that the curve is embedded. The fact that any curve with a point with large multiplicity has to have large energy $\CalE\CalL$ has already been observed in \cite{Polden,vonderMosel,WheelerCurveDiffusion,Wojtowytsch,PozzettaVarifold}. More precisely, one may bound the energy by $\CalE(\gamma)\CalL(\gamma)\geq c k^2$  whenever $\gamma$ possesses a point with multiplicity $k\in \N$. Here $c>0$ is a constant and as a consequence of \cite{Polden,vonderMosel,WheelerCurveDiffusion} $c=16$ is possible. While such a relation between multiplicity and energy is also part of the statement of the original Li--Yau inequality for the Willmore energy \cite{LY82}, this does not give the optimal threshold to guarantee embeddedness.
%Genauere Referenzen/Konstanten: \cite[Theore 3.3.1.4.]{Polden} \cite[Theorem 4.3]{vonderMosel} and \cite[Theorem 1.6]{WheelerCurveDiffusion} for $c=16$, \cite[Lemma 2.1]{Wojtowytsch} for $c=\pi^2$ and $c=4}	
%	Note that it was already observed in \cite[Lemma 2.1]{Wojtowytsch}, that if $\gamma \in W^{2,2}(\S^1;\R^n)$ has a point with multiplicity $k\in \N$, then $\CalE(\gamma)\CalL(\gamma) > \pi^2k^2$. 
%	 However, we remark that the bound in \cite{Wojtowytsch} cannot be used to ensure embeddedness, since any curve $\gamma$ with $\CalE(\gamma)\CalL(\gamma)\leq 4\pi^2$ has to be a one-fold covered circle, cf. \Cref{lem:charelas}. Moreover, with the methods of geometric measure theory, in \cite[(16)]{PozzettaVarifold} the bound $k\leq \frac{1}{2} \CalK(\gamma)$ was established. As $\CalK(\gamma)\geq 2\pi$, this estimate cannot answer the question of embeddedness either.
	
	The idea for proving \Cref{thm:main} is to look at the minimization problem

	\begin{align}\label{eq:gelato}
		\inf\{\CalE(\gamma)\CalL(\gamma)\mid \gamma\in W^{2,2}(\S^1;\R^2) \text{ non-embedded immersion}\}.
	\end{align}
	
	Minimizing among non-embedded immersions is a non-standard condition, because the admissible set is not open. This causes difficulties in applying Euler--Lagrange methods. However, we will be able to deduce that the minimizer is an interior point of the admissible set and thus satisfies an Euler--Lagrange equation. This can be achieved by a detailed analysis of the self-intersections of minimizers. The main ingredient here is the classification of planar elastic curves (see for instance \cite{LangerSinger1,DHV,AnnaNetworks}).

 As a future extension of \Cref{thm:main} one could also try to find such \textit{embeddedness-ensuring quantities} in other ambient manifolds than $\mathbb{R}^2$.
\\
 In  the hyperbolic half-plane $\mathbb{H}^2$ an embeddedness-ensuring quantity can indeed be identified. By \cite{LangerSinger2} one has for all  immersed curves $\gamma \in C^\infty( \mathbb{S}^1; \mathbb{H}^2)$ 
 \begin{equation}\label{eq:ElasticVSWillmore}
 \int_{\S^1} |\kappa_{\mathbb{H}}[\gamma]|^2 \; \mathrm{d}s = \frac{2}{\pi} \mathcal{W}( S(\gamma) ) , 
 \end{equation}
 
 where $\kappa_{\mathbb{H}}[\gamma]$ denotes the hyperbolic curvature of $\gamma$ and $S(\gamma)$ denotes the immersion that arises from revolution of $\gamma$  around the $x_1$-axis. This and \eqref{eq:LYWillmore} yield that
 \begin{equation}
 \int_{\S^1} |\kappa_{\mathbb{H}}[\gamma]|^2 \; \mathrm{d}s  < 16 \; \; \textrm{implies that $\gamma$ is an embedding}.
\end{equation}  
The threshold of $16$ is also sharp for this implication, cf. \cite[Corollary 6.4]{MS20}. Notice that this does not immediately follow from the sharpness of the inequality in  \eqref{eq:LYWillmore}. Indeed, the standard examples which yield sharpness of the classical Li--Yau inequality are not necessarily surfaces of revolution which arise from revolving a closed curve. Hence \eqref{eq:ElasticVSWillmore} can not be used to obtain sharpness of the
	threshold of 16 immediately.
 \\
In $\mathbb{S}^2$ the elastic energy of curves $\gamma\in C^{\infty}(\S^1;\S^2)$ given by $$  \gamma \mapsto \int_{\S^1} |\kappa_{\mathbb{S}^2}[\gamma]|^2 \; \mathrm{d}{s}$$ is not an embeddedness-ensuring quantity since any two-fold cover of a closed geodesic in $\mathbb{S}^2$ is non-embedded and has vanishing energy.  
 \\
  It would be interesting to investigate whether Theorem \ref{thm:main} generalizes to curves in $\mathbb{R}^3$, see Remark \ref{rem:5.12} for some ideas in this context. With the gradient flow methods in \Cref{sec:flow} one could then deform any curve with sufficiently small energy into an elastica, while preserving its knot class; an observation that was already made in \cite[Chapter 3.4]{Polden}. However, even if \Cref{thm:main} generalizes to higher codimension, the corresponding energy threshold is necessarily less then or equal to $c^*$, as we can of course view $\gamma^*$ as a spatial curve. On the other hand, as a consequence of the Fáry--Milnor Theorem \cite{Fary,MilnorKnots}, any curve $\gamma$ with $\CalE(\gamma)\CalL(\gamma)<c^*$ has to be unknotted, see \Cref{rem:5.12} below.
	
\section{Notational preliminaries}
In the following, we will view the $1$-sphere  as $\mathbb{S}^1 =[0,1]/{\sim}$, where $\sim$ denotes the equivalence relation that identifies $0 \sim 1$ and all other points only with themselves. Equivalently,  $\mathbb{S}^1\cong\R/ \Z$. 
Consequently, an interval $[a,b]\subset \mathbb{S}^1$ with $a<b$ has to be understood with respect to this equivalence relation, i.e. $[a,b] = \{ [x]_{\sim} : x \in [a,b] \}$. % or equivalently $[a,b] = \{ x+\Z : x \in [a,b] \}\subset \R/\Z$. 
For the sake of simplicity of notation, we define the interval $[b,a]\defeq [b,1]\cup [0,a]$ for $a,b\in [0,1]$ with $a<b$. In the same fashion, the open and half-open intervals are defined.

%In the following, we will frequently identify $\mathbb{S}^1$ with $[0,1] / \hspace{-0.15cm} \sim$, where $\sim$ denotes the equivalence relation that identifies $0 \sim 1$ and all other points only with themselves. Note that $\mathbb{S}^1= \mathbb{R}/ \hspace{-0.15cm} \sim'$ where $\sim'$ identifies each real number with those that differ from it by an integer. If we write intervals $[a,b],a<b$ we always think of them as subject to the equivalence relation $\sim'$, i.e. $[a,b] = \{ [x]_{\sim'} : x \in [a,b] \}$. The same one can do for open and half-open intervals. For $a,b \in [0,1]$ such that $a< b$ we also define for the sake of simplicity of notation $[b,a]:= [b,1] \cup [0,a]$ where both intervals on the right hand side of the definition underlie the convention above. 
\begin{defi}
We define for $k, \ell \in \mathbb{N}, p \in [1,\infty]$ the \emph{Sobolev space} $W^{k,p}(\mathbb{S}^1;\R^\ell)$ {as} %to be 
\begin{equation}
W^{k,p}(\mathbb{S}^1;\R^\ell) \defeq \{ u \in W^{k,p}((0,1);\R^\ell) \mid u^{(m)}(0) = u^{(m)}(1) \, \forall m = 1,...,k-1\},
\end{equation}
where $u^{(m)}$ denotes the continuous representative of the $m$-th weak derivative. Moreover, for $k\geq 2$ we denote by $W^{k,p}_{Imm}(\S^1;\R^\ell)$ the set of $W^{k,p}$-immersions.
\end{defi} 
\begin{remark}\label{rem:closedness}
It can be seen that this definition coincides with the general definition of Sobolev spaces on manifolds, cf. \cite[Definition 2.1]{Hebey}. This is why we can also use general results about these spaces and also talk about Sobolev spaces on open subsets of $\mathbb{S}^1$. We will refer to curves in $C^k(\mathbb{S}^1;\mathbb{R}^2)$ as $C^k$-closed, which is also due to the fact that
{
\begin{equation}
C^k(\mathbb{S}^1;\mathbb{R}^2) = \{ u \in C^k([0,1];\mathbb{R}^2) \mid u^{(m)}(0) = u^{(m)}(1) \; \forall m = 0,...,k \}.
\end{equation}
}In particular, each curve $W^{k,p}(\mathbb{S}^1;\mathbb{R}^2)$ is $C^{k-1}$-closed. Observe also that each curve in $C^k(\mathbb{S}^1; \mathbb{R}^2) $ possesses an extension to a $1$-periodic curve in $C^k(\mathbb{R},\mathbb{R}^2)$.
\end{remark}
\begin{remark}\label{rem:13}
Another noticeable property of $W^{k,p}(\mathbb{S}^1)$ is gluing, i.e. if $u \in W^{k,p}((a,b))$ and $v \in W^{k,p}((b,a))$ are such that $u^{(m)}(a)=v^{(m)}(a)$ and $u^{(m)}(b) = v^{(m)}(b)$ for all $m = 0,..., k -1$ then 
\begin{equation}
w(x) \defeq \begin{cases} 
u(x) & x \in (a,b), \\ v(x) & x \in (b,a)
\end{cases}
\end{equation}
lies in $W^{k,p}(\mathbb{S}^1)$. 
\end{remark}

We now review some basic geometric definitions of planar curves. For an immersion $\gamma\colon \S^1\to \R^2$ we write $\gamma(x)=(\gamma_1(x), \gamma_2(x)), x\in \S^1$ for the components and $\gamma^{\prime}=\partial_x \gamma$ for the derivative. Moreover, we write $\kappa\defeq {\abs{\gamma^{\prime}}^{-3}}{\det\left(\gamma^{\prime}, \gamma^{\prime\prime}\right)}$ for its \emph{(signed) curvature}. Another important geometric object is the \emph{arc-length derivative}, denoted by $\partial_s = \abs{\gamma^{\prime}}^{-1}\partial_x$ and the \emph{arc-length element} $\diff s \defeq \abs{\gamma^{\prime}}\diff x$. The \emph{curvature vector field} is $\vKap = \partial_s^2 \gamma = \kappa \vec{n}$, where $\vec{n}$ denotes the \emph{unit normal}, obtained by rotating $\partial_s \gamma$ counterclockwise by $\frac{\pi}{2}$.

\section{A non-existence result for the total curvature}
In this section, we will prove \Cref{thm:infTotalCurv} and show why it implies that there is no non-trivial generalization of \eqref{eq:LYWillmore} involving the total curvature.

\begin{proof}[{Proof of \Cref{thm:infTotalCurv}}]
	By Fenchel's theorem, cf. \Cref{thm:Fenchel}, we have $\CalK(\gamma)\geq 2\pi$ for all $\gamma\in C^{2}(\S^1;\R^2)$. Consequently
	\begin{align}\label{eq:InfTotalnon-embedded}
	\inf \left\lbrace\CalK(\gamma)\mid \gamma\in C^{2}(\S^1;\R^2) \text{ non-embedded immersion} \right\rbrace \geq 2\pi.
	\end{align}
	To prove equality, we take some angle $\beta\in (0,\frac{\pi}{2})$. First, take an arc $S_\beta$ of a circle of radius 1 of length $2\pi-2\beta$ and place it symmetrically with respect to the $x_1$-axis. Extend the segment by the tangent lines at its endpoints. For $\beta\in (0,\frac{\pi}{2})$, they will intersect in a point on the $x_1$-axis. Reflecting everything with respect to that point gives a closed curve $\gamma$ which is not embedded, cf. \Cref{fig:1a}.
	
	\begin{figure}[h]
		\centering
		\subfloat[The construction with $\beta=\frac{\pi}{4}$]{\label{fig:1a}\includegraphics[height=2cm]{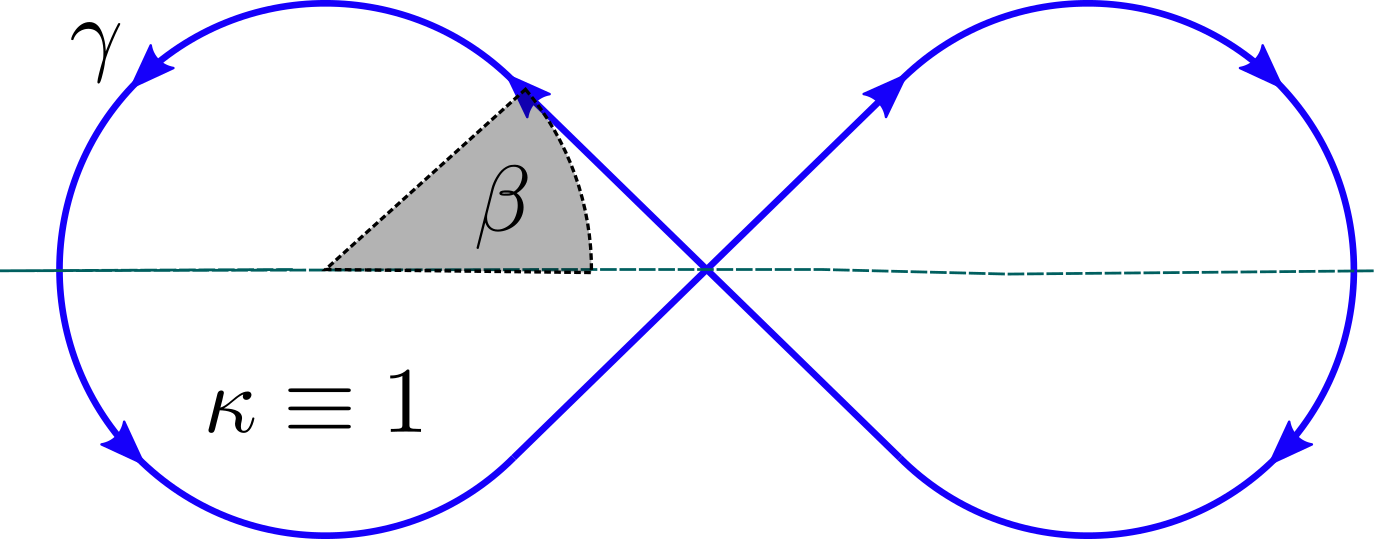}}
		\qquad 
		\subfloat[The curve becomes longer as $\beta\nearrow\frac{\pi}{2}$]{\label{fig:1b}\includegraphics[height=2cm]{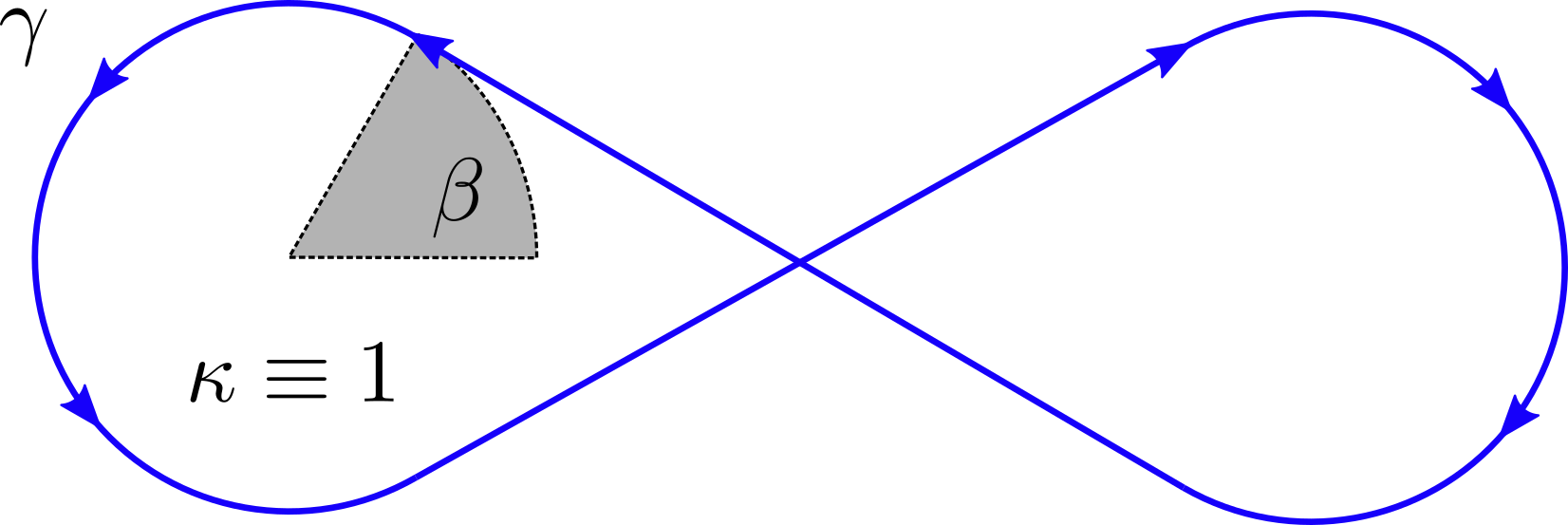}}
		\caption{A non-embedded $W^{2,2}$-curve with approximate total curvature $2\pi$ as $\beta\nearrow\frac{\pi}{2}$.}
	\end{figure}
	
	By the glueing property in \Cref{rem:13}, we have $\gamma\in W^{2,2}(\S^1;\R^2)$. 
	In order to compute $\CalK(\gamma)$ we first note that $\abs{\kappa}\equiv 1 $ on the left circle segment, whereas $\kappa\equiv 0$ on the straight line. Thus by symmetry
	\begin{align}
	\CalK(\gamma) = 2 \int_{S_{\beta}}\diff s = 2(2\pi-2\beta) = 4\pi - 4\beta \searrow 2\pi, \text{ as } \beta\nearrow \frac{\pi}{2}. 
	\end{align}
Using the characterization of the equality in Fenchel's theorem (\Cref{thm:Fenchel}), we may conclude that the infimum in \eqref{eq:InfTotalnon-embedded} is not attained by a non-embedded immersion.
\end{proof}

\Cref{thm:infTotalCurv} shows that $\CalK$ cannot distinguish between embedded and non-embedded immersions, since the least possible energy among all closed curves can be approximated by non-embedded ones. 
	This justifies that \eqref{eq:LYWillmore} has no non-trival generalization to the total curvature.

%Now, suppose there is $C>0$ such that for all immersions $\gamma\in C^{2}(\S^1;\R^2)$ we have
%\begin{align}\label{eq:LYTotalCurv}
%\CalK(\gamma)<C \text{ implies that } \gamma \text{ is an embedding.}
%\end{align}
%Consequently, we get that $\CalK(\gamma)\geq C$ for all non-embedded immersions $\gamma\in C^{2}(\S^1;\R^2)$, hence $2\pi\geq C$ by \Cref{thm:infTotalCurv}. But, by Fenchel's theorem $\CalK(\gamma)\geq 2\pi$ for all curves $\gamma\in C^{2}(\S^1;\R^2)$. Thus, there does not exist any $\gamma$ satisfying $\CalK(\gamma)<C$ making \eqref{eq:LYTotalCurv} vacuously true. 

%On the other hand, if we allow equality in \eqref{eq:LYTotalCurv}, it is true that any curve $\gamma\in C^2(\S^1;\R^2)$ with total curvature $\CalK(\gamma)=2\pi$, has to be embedded (and even convex) by \Cref{thm:Fenchel}. This justifies that \eqref{eq:LYWillmore} has no nontrival generalization to the total curvature.

\section{The variational problem and existence of a minimizer}
In order to prove \Cref{thm:main}, we wish to minimize the functional $\CalE(\gamma)\CalL(\gamma)$ among all curves  $\gamma\in W_{Imm}^{2,2}(\S^1;\R^2)$ which are not embeddings. % \changeMar{We will identify the minimizer and infer that values below t} 

 As a first step, we want to {characterize embeddedness in a way that is useful for variational discussions}.
\begin{lem}\label{lem:embeddVSInjective}
	An immersed curve $\gamma\in C^1(\mathbb{S}^{1},\R^{2})$ is an embedding, if and only if $\gamma$ is injective. 
\end{lem}
\begin{proof}
	See for instance \cite[Proposition 5.4]{LeeSM}.
\end{proof}

This implies the following very useful characterization of embeddedness.

\begin{lem}[Characterization of Embeddedness]\label{lem:EmbeddChara}
	Let $\gamma\in C^1(\mathbb{S}^{1};\R^{2})$ be an immersion. Then $\gamma$ is embedded if and only if  $A[\gamma]\defeq \inf_{x\neq y} \frac{\abs{\gamma(x)-\gamma(y)}}{\abs{x-y}}>0$. 
\end{lem}
\begin{proof}
	Suppose $A[\gamma]>0$. Then, we have $\abs{\gamma(x)-\gamma(y)}\geq A[\gamma]\abs{x-y} >0$ for $x\neq y\in \mathbb{S}^{1}$. Thus $\gamma$ is injective, hence an embedding by \Cref{lem:embeddVSInjective}.
	
	Conversely, suppose $\gamma$ is an embedding and $A[\gamma]=0$. Then, there exist $x_n \neq y_n$ such that $\frac{\abs{\gamma(x_n)-\gamma(y_n)}}{\abs{x_n -y_n}}\to 0$. Passing to a subsequence, we have $x_n \to x, y_n \to y$ for $x,y\in \mathbb{S}^{1}$ by compactness. If $x\neq y$, we have
	\begin{align}\label{eq:gammaxgammay}
		0 =\lim_{n \to \infty}\frac{\abs{\gamma(x_n)-\gamma(y_n)}}{\abs{x_n -y_n}} = \frac{\abs{\gamma(x)-\gamma(y)}}{\abs{x-y}},
	\end{align}
	hence $\gamma(x) = \gamma(y)$. This is a contradiction to the embeddedness of $\gamma$, cf. \Cref{lem:embeddVSInjective}. Hence $x=y$ in $\S^1$ and for $i=1,2$ we have
	\begin{align}
		\gamma_i(x_n)-\gamma_i(y_n) = \gamma_i^{\prime}(\xi_{i,n})(x_n-y_n)
	\end{align}
	for some $\xi_{i,n}\in \S^1$ between $x_n$ and $y_n$. Dividing by $x_n-y_n$ and using the assumption, we find
	\begin{align}
		\gamma_i^{\prime}(x)=\lim_{n\to\infty} \gamma_i^{\prime}(\xi_{i,n}) = \lim_{n\to\infty} \frac{\gamma_i(x_n)-\gamma_i(y_n)}{x_n-y_n} = 0,
	\end{align}
	for $i=1,2$, a contradiction to $\gamma$ being an immersion.
%	
%	\begin{align}
%		\abs{\gamma(x_n)-\gamma(y_n) - \gamma^{\prime}(x)(x_n-y_n)} 		&\leq \sup_{z\in [x_n, y_n]}\abs{\gamma^{\prime}(z)-\gamma^{\prime}(x)} \abs{x_n-y_n}.
%	\end{align}
%	Given any $\varepsilon>0$, for $n\geq N$ large enough, we have $\sup_{z\in [x_n, y_n]}\abs{\gamma^{\prime}(z)-\gamma^{\prime}(x)} <\varepsilon$ since $\gamma\in C^{1}(\mathbb{S}^{1},\R^{2})$. Therefore, 
%	\begin{align}
%		\frac{\abs{\gamma(x_n)-\gamma(y_n) - \gamma^{\prime}(x)(x_n-y_n)}}{\abs{x_n -y_n}} \to 0.		
%	\end{align}
%	In particular, using \eqref{eq:gammaxgammay}, we obtain $\abs{\gamma^{\prime}(x)}=0$, a contradiction.
%	\begin{align}
%		0 =\lim_{n\to\infty}\frac{\abs{\gamma(t_n)-\gamma(s_n)}}{\abs{t_n-s_n}} = \lim_{n\to\infty} \abs{\frac{{\gamma(t_n)-\gamma(s_n) - \gamma^{\prime}(t)(t_n-s_n)}}{\abs{t_n -s_n}}\pm \gamma^{\prime}(t)} \to \abs{\gamma^{\prime}(t)},
%	\end{align}
%	a contradiction to $\gamma$ being an immersion.
\end{proof}
An important consequence is the following lemma.
\begin{lem}\label{lem:embeddingsopen}
	The set of $C^1$-embeddings is an open subset of $C^{1}(\mathbb{S}^{1};\R^{2})$.
\end{lem}
\begin{proof}
	Suppose $\gamma\in C^{1}(\mathbb{S}^{1};\R^{2})$ is an embedding. We claim that for  $\varepsilon>0$ small enough any $\tilde{\gamma}\in C^{1}(\S^1;\R^2)$ with $\Norm{\tilde{\gamma}-\gamma}{C^1} <\varepsilon$ is an embedding. By \Cref{lem:EmbeddChara}, it suffices to show $A[\tilde{\gamma}]>0$, since $\tilde{\gamma}$ is clearly an immersion for $\varepsilon>0$ small enough. We have for $x\neq y\in \S^1$
	\begin{align}
		\frac{\abs{\tilde{\gamma}(x)-\tilde{\gamma}(y)}}{\abs{x-y}} &\geq  \frac{\abs{\gamma(x)-\gamma(y)} - \abs{\tilde{\gamma}(x)-\gamma(x)-(\tilde{\gamma}(y))-\gamma(y))}}{\abs{x-y}} \\
		&\geq A[\gamma] - \sup_{x\neq y}\frac{\abs{\tilde{\gamma}(x)-\gamma(x)-(\tilde{\gamma}(y))-\gamma(y))}}{\abs{x-y}} \\
		&\geq A[\gamma] - \Norm{\tilde{\gamma}-\gamma}{C^{1}} \\
		&\geq  A[\gamma]- \varepsilon >0
	\end{align}
	if $\varepsilon>0$ is small enough.
\end{proof}

	\begin{remark}\label{rem:embeddings ab}
		In the proofs of \Cref{lem:EmbeddChara} and \Cref{lem:embeddingsopen}, we did not really use the specific structure of $\S^1$. In particular, the statements of \Cref{lem:EmbeddChara} and \Cref{lem:embeddingsopen} remain true if one replaces $\S^1$ by any compact interval $[a,b]\subset \R$.
	\end{remark}

We now consider a minimization problem, cf. \eqref{eq:gelato}. Define the set
\begin{align}
\mathcal{A} \defeq \{ \gamma\in W^{2,2}(\S^1;\R^2)\mid \gamma\text{ is a non-injective immersion}\}\subset W^{2,2}(\S^1;\R^2).
\end{align}
\begin{thm}[Existence of a Minimizer]\label{lem:variational}
	There exists $\bar{\gamma}\in \CalA$ with
	\begin{align}\label{eq:infprob}
		\CalE(\bar{\gamma})\CalL(\bar{\gamma}) = \inf_{\gamma\in \mathcal{A}} \CalE(\gamma)\CalL(\gamma)>0.
	\end{align}
\end{thm}
\begin{proof}
	We consider the set $\tilde{\mathcal{A}}\defeq \{\gamma\in \mathcal{A}\mid \CalL(\gamma)=1\}\neq \emptyset$. As a first step, we show that there exist $\bar\gamma\in \tilde{\mathcal{A}}$ with 
	\begin{align}\label{eq:constVariationalProblem}
		\CalE(\bar{\gamma}) = \inf_{\gamma\in \tilde{\mathcal{A}}} \CalE(\gamma).
	\end{align} 
	Let $\left(\gamma^{(n)}\right)_{n\in \N}$ be a minimizing sequence for \eqref{eq:constVariationalProblem}. Without loss of generality, we may assume $\gamma^{(n)}$ to be parametrized by arc-length for all $n\in \N$, cf. \Cref{lem:reparaConstSpeed}. By reflexivity and the compactness of the embedding $W^{2,2}(\S^1;\R^2)\hookrightarrow C^1(\S^1;\R^2)$, we have $\gamma^{(n)}\rightharpoonup \bar\gamma$ in $W^{2,2}(\S^1;\R^2)$ and $\gamma^{(n)} \to \bar\gamma$ in $C^1(\S^1; \R^2)$ for some $\bar\gamma\in W^{2,2}(\S^1;\R^2)$, passing to a subsequence. We have $\abs{\gamma^{\prime}_n(x)} = 1$ for all $x\in \S^1, n \in \N$, thus $\bar{\gamma}$ is parametrized with unit speed and $\CalL(\bar\gamma)=1$. Moreover, by \Cref{lem:embeddingsopen}, $\bar\gamma$ cannot be an embedding as $\gamma^{(n)}\to \bar{\gamma}$ in $C^1(\S^1;\R^2)$. Consequently we have $\bar\gamma\in \tilde{\mathcal{A}}$ by \Cref{lem:embeddVSInjective}.
	
	For unit speed curves $\gamma\in W^{2,2}(\S^1;\R^2)$, the elastic energy is given by
	\begin{align}
		\CalE(\gamma) = \int_{\S^1}\abs{\gamma^{\prime\prime}(x)}^2 \diff x.
	\end{align}
	Since the $L^2(\S^1;\R^2)$-norm is weakly lower semicontinuous, we conclude
	\begin{align}
		\CalE(\bar \gamma)\leq \liminf_{n\to\infty} \CalE(\gamma_n) = \inf_{\gamma\in \tilde{\mathcal{A}}} \CalE(\gamma),
	\end{align}
	so $\bar \gamma$ is a minimizer. Since $\R^2$ does not allow for closed geodesics one infers that $\CalE(\bar{\gamma})>0$.
	
	Now, if $\gamma\in W^{2,2}(\S^1;\R^2)$ is an arbitrary immersion, we may consider the rescaled curve $\tilde{\gamma}\defeq \frac{1}{\CalL(\gamma)} \gamma$, so $\tilde{\gamma}\in \tilde{\mathcal{A}}$. Note that $\gamma$ is an embedding if and only if $\tilde{\gamma}$ is an embedding.
	Consequently, we have
	\begin{align}
		\CalE(\gamma)\CalL(\gamma) = \CalE(\tilde{\gamma}) \geq \CalE(\bar{\gamma}) = \CalE(\bar{\gamma})\CalL(\bar{\gamma}),
	\end{align}
	and hence $\inf_{\gamma\in \mathcal{A}}\CalE(\gamma)\CalL(\gamma) = \CalE(\bar{\gamma})\CalL(\bar{\gamma})$.
\end{proof}

\section{The Euler-Lagrange equation}
%We have already shown that there exists $\overline{\gamma} \in W^{2,2}(\mathbb{S}^1, \mathbb{R}^2)$ that satisfies 
%\begin{equation}\label{eq:infprob}
%\CalE(\bar{\gamma}) \CalL(\bar{\gamma}) = \inf_{\gamma \in \mathcal{A}}  \CalE(\gamma) \CalL(\gamma) ,
%\end{equation} 
%where 
%\begin{equation}
%\mathcal{A} = \{ \gamma \in W^{2,2}(\mathbb{S}^1;\mathbb{R}^2) : \gamma \; \textrm{non-injective immersion} \} \subset W^{2,2}(\mathbb{S}^1,\mathbb{R}^2).
%\end{equation}
{In this section, we will study the properties of a minimizer $\bar{\gamma}$ from \Cref{lem:variational} in order to prove our main theorem.}
%Since $\CalE(\bar{\gamma}) \CalL(\bar{\gamma})=C$ is also the constant in Theorem \ref{thm:main} which we intend to determine, we study properties of $\bar{\gamma}$ here.
 The most important property we will derive is that $\overline{\gamma}$ is a \emph{constrained elastica}, i.e. $\bar{\gamma} \in C^\infty(\mathbb{S}^1;\mathbb{R}^2)$ and 
\begin{equation}\label{eq:elasticaeq}
\partial_s^2\kappa + \frac{1}{2} \kappa^3 - \lambda \kappa = 0 \quad \text{for some }\lambda \in \R.
\end{equation}
Solutions of the \emph{constrained elastica equation} have been classified in previous works, eg. by \cite{LangerSinger1,DHV}. What one needs to show for this is that $\bar{\gamma}$ satisfies the \emph{Euler--Lagrange equation}, i.e. for all $\phi \in C^\infty(\mathbb{S}^1;\mathbb{R}^2)$ one has 
\begin{equation}\label{eq:EulerLag}
\CalL(\bar{\gamma})D\CalE(\bar{\gamma})(\phi) + \CalE(\bar{\gamma}) D\CalL(\bar{\gamma})(\phi) = 0.
\end{equation}
Indeed, as one can see following the lines of \cite[Section 5]{EichmannGr}, equation \eqref{eq:EulerLag} implies that $\overline{\gamma}$ is smooth and  \eqref{eq:elasticaeq} holds. If $\bar{\gamma}$ is an {\emph{interior point}} of $\mathcal{A}$ then \eqref{eq:EulerLag}  follows from the fact that for all $\phi \in C^\infty(\mathbb{S}^1;\mathbb{R}^2)$  one has 
\begin{equation}\label{eq:variation}
0 = \frac{\diff}{\diff \varepsilon}\Bigg\vert_{\varepsilon= 0} \mathcal{E}(\bar{\gamma}+ \varepsilon \phi) \mathcal{L}(\bar{\gamma} + \varepsilon \phi),
\end{equation}
{since $(-\varepsilon_0, \varepsilon_0)\ni \varepsilon\mapsto\CalE(\bar{\gamma}+\varepsilon\phi)\CalL(\bar{\gamma}+\varepsilon\phi)$ attains a minimum at $\varepsilon=0$.}
The problem is that we do not know to begin with whether the infimum is attained at an interior point of $\mathcal{A}$.
The goal of this section is to prove exactly this.
 \begin{figure}[h]
 	\centering
 	\includegraphics[width=0.5\textwidth]{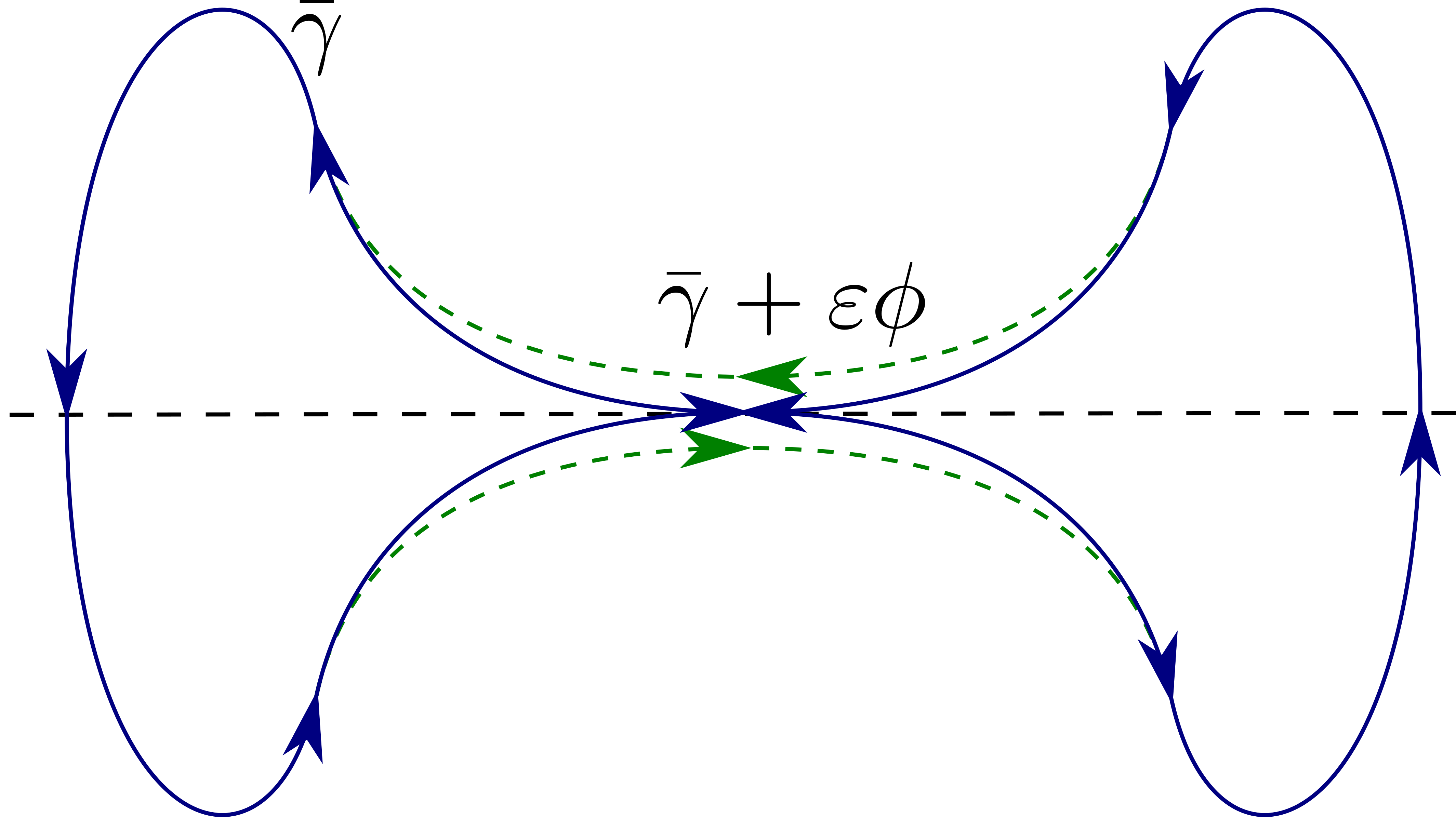}
 	\caption{A tangential self-intersection vanishing under small perturbations.}
 	\label{fig:tangentialSelfIntersection}
 \end{figure}
 
 We remark that $\mathcal{A}\subset W^{2,2}(\S^1;\R^2)$ is not open. Indeed, like in \Cref{fig:tangentialSelfIntersection} it is possible to leave the set $\mathcal{A}$ by a variation that eliminates self-intersections. To show that {each minimizer} is an interior point we need to examine these %those
  self-intersections, denoted by 
 \begin{equation}
 S[\gamma] \defeq \{ p \in \mathbb{R}^2 : \exists  \; x_1 \neq x_2 \; s.t.  \; \gamma(x_1) = \gamma(x_2) = p \}.
 \end{equation}
 We also define the multiplicity of $p \in S[\gamma]$ to be $\mathrm{mult}[\gamma](p) \defeq \mathcal{H}^0(\gamma^{-1}(\{p\}))$ where $\mathcal{H}^0$ denotes the counting measure. Moreover we have to pay special attention to the tangential self-intersections given by
 \begin{equation}
 S_{tan}[\gamma] \defeq \{ p \in \mathbb{R}^2 : \exists \; x_1 \neq x_2  \; s.t. \; \gamma(x_1) = \gamma(x_2) = p, \; \mathrm{det}(\gamma'(x_1),\gamma'(x_2)) =0 \}.
 \end{equation}
Before we can proceed with the proof we need some preparations that exclude certain configurations.
\subsection{Some facts about constrained elasticae}\label{sec:31}
In this section, we will discuss properties of closed elasticae. We remark that at this point we do not know whether the minimizer $\bar{\gamma}$ is a constrained elastica. However, the analysis of the energy and the self-intersection properties of closed elasticae will play a crucial role in proving \Cref{thm:main} later. More precisely, we will be able to exclude self-intersection properties of minimizers $\bar{\gamma}$ once we can exclude them for non-embedded constrained elasticae.
Define
\begin{equation}\label{eq:defB}
\mathcal{B}\defeq \{ \gamma \in C^\infty(\mathbb{S}^1;\mathbb{R}^2)  : \gamma \; \textrm{is a non-embedded constrained elastica} \} \subset \mathcal{A}.
\end{equation}
Throughout this section we will use the elliptic functions defined in Appendix \ref{appendix:Jacobi}. 
We next define one special elastica, which will be important for our considerations.
\begin{defi}[The Elastic Figure Eight]\label{def:figeight}
Let $m^* \in (0,1)$ be the unique root of the map $(0,1) \ni m \mapsto 2 E(m)- K(m)$, see Lemma \ref{lem:uniqueroot}. We define the \emph{one-fold cover of the elastic figure eight} by
\begin{equation}\label{eq:fiigeight}
\gamma^* :[0,4K(m^*)] \rightarrow \mathbb{R}^2 , \quad  \gamma^*(s) \defeq \begin{pmatrix}
2 E(\am(s,m^*), m^*) - s \\ -2 \sqrt{m^*} \cn(s,m^*) 
\end{pmatrix}.
\end{equation}
\end{defi}
\begin{remark}\label{rem:value m*}
	The numerical value of $m^*$ is approximately $0.8261$. This is not needed in the sequel, since all computations we provide are analytical. However, it enables us to approximate the value of $c^*$.
\end{remark}
\begin{remark}\label{rem:3.2}
We will also look at another parametrization of the figure eight that is easier for computations e.g. to compute self-intersections. For this observe that by \eqref{eq:fiigeight} and Appendix \ref{appendix:Jacobi}
\begin{equation}
\gamma^*(F(x,m^*)) = \begin{pmatrix}
2 E(x, m^*) - F(x,m^*) \\ -2 \sqrt{m^*} \cos(x) 
\end{pmatrix}.
\end{equation}
and note that $F(\cdot,m^*)$ is strictly monotone with $F(0,m^*) = 0$ and $F(2\pi,m^*) = 4 K(m^*)$. Hence 
\begin{equation}
\tilde{\gamma}^{\ast} : [0,2\pi] \rightarrow \mathbb{R}^2 , \quad \tilde{\gamma}^{\ast}(x) \defeq  \begin{pmatrix}
2 E(x, m^*) - F(x,m^*) \\ -2 \sqrt{m^*} \cos(x) 
\end{pmatrix}
\end{equation} 
is also a parametrization of the figure eight{.}%, not necessarily arc-length parametrized. 
\end{remark}
%We will later show that the figure eight is smoothly closed, i.e. it has a reparametrization in $C^\infty(\mathbb{S}^1; \mathbb{R}^2)$, cf. Remark \ref{rem:closedness}. It will also be a minimizer for our variational problem. Already a a hint in this direction is
We show next that the elastic figure eight is smoothly closed and minimizes our functional in $\mathcal{B}$. It will turn out later that it is actually also a minimizer in $\mathcal{A}$.

\begin{lem}[Characterization of Closed Elasticae]\label{lem:charelas}
The only closed constrained elasticae are (possibly rescaled, rotated, translated and reparametrized versions of) multi-fold coverings of circles and multi-fold coverings of the figure eight. All of these elasticae are not embedded except for the one-fold covering of the circle. Moreover,
\begin{equation}\label{eq:infprob_constrElastica}
\inf_{\gamma \in \mathcal{B}}  \mathcal{E}(\gamma) \mathcal{L}(\gamma) = \mathcal{E}(\gamma^*) \mathcal{L}(\gamma^*),
\end{equation}
with $\gamma^{\ast}$ as in \Cref{def:figeight}. Equality holds if and only if $\gamma$ is a rescaled, translated and rotated reparametrization of $\gamma^*$.
\end{lem}
\begin{proof}
First we show the assertion that the only closed elasticae are given by the figure eight and the circle. By Proposition \ref{prop:elaclassi} there are --- up to scaling, isometries in $\mathbb{R}^2$, and reparametrization --- only five different types of elasticae which we all examine separately for closedness. %For the sake of simplicity of notation we will denote the first component of a curve $\gamma$ always by $x$ and the second component by $y$. 

\textbf{Type 1: Linear elasticae.} Since lines are not closed they cannot generate closed elasticae. 

\textbf{Type 2: Wavelike elasticae.} For closedness it is necessary that both components of $\gamma$ are periodic with the same period $L$, cf. Remark \ref{rem:closedness}. The period itself does not matter since we can always reparametrize the curve. In the wavelike case we have by Proposition \ref{prop:elaclassi} 
\begin{align}
\gamma_1(s) & = 2 E(\am(s,m),m) - s  & \gamma_2(s) = - 2\sqrt{m} \cn(s,m) .
\end{align}
By Proposition \ref{prop:identities} all periods $L$ of $\gamma_2$ are given by 
\begin{equation}
L \in \{4 l K(m) : l \in \mathbb{N} \}.
\end{equation}
We investigate with the aid of Proposition \ref{prop:identities} for which values of $m$ one of these periods is also a period of $\gamma_1$. For $l \in \mathbb{N}$ we compute 
\begin{align}
\gamma_1(s+ 4l K(m) ) & = 2 E(\am(s,m) + 2 l \pi , m )- s - 4lK(m) \\ &  = 2 E(\am(s,m),m) + 8 l E(m)  - s - 4l K(m) \\
&= \gamma_1(s) + 4l( 2E(m) - K(m)). \label{eq:periwave}
\end{align}
Hence $\gamma_1$ and $\gamma_2$ share a period if and only if $2E(m)-K(m) = 0$. In this case $4K(m)$ is already a joint period of $\gamma_1$ and $\gamma_2$, and hence a period of $\gamma$. By Lemma \ref{lem:uniqueroot}, $2E(m) -K(m)$ has only one zero $m^* \in(0,1)$, which by Definition  \ref{def:figeight} yields exactly the elastic figure eight. 

\textbf{Type 3: Borderline elastica.} This can not be periodic since the second component has no real period. 

\textbf{Type 4: Orbitlike elasticae.} We proceed similar as in the wavelike case. Recall that by Proposition \ref{prop:elaclassi}
\begin{align}
\gamma_1(s) & = \frac{2}{m} E(\am(s,m),m) + \left( 1 - \frac{2}{m} \right) s,  & \gamma_2(s) = - \frac{2}{m} \dn(s,m) .
\end{align}
By Proposition \ref{prop:identities} all periods %of
$L$ of $\gamma_2$ are given by 
\begin{equation}
L \in \{ 2 l K(m) : l \in \mathbb{N} \}
\end{equation}
Next we look at the behavior of $\gamma_1$, which we can characterize by Proposition \ref{prop:identities} to be 
\begin{align}
\gamma_1(s + 2 lK(m) ) = \gamma_1(s) +  l \frac{4E(m) - 4K(m) + 2m K(m) }{m}.
\end{align}
Hence $\gamma_1$ and $\gamma_2$ share a period if and only if $4E(m) - 4K(m) + 2m K(m)=0$, which is impossible because of Lemma \ref{lem:4.5}. Hence there do not exist closed orbitlike elasticae. 

\textbf{Type 5: Circular elasticae.} Circles are trivially closed. Their period is given by $2\pi$ when using the standard arclength parametrization $x \mapsto (\cos(x), \sin(x))$. 

Summarizing our findings for all types, we obtain that the only closed elasticae are circles and the figure eight as well as multiple coverings of these. We have also found their periods. Next we show that the one-fold cover of the figure eight is not injective, so that it is admissible for the minimization problem in the statement. For this let $\gamma^* = (\gamma_1^{\ast},\gamma_2^{\ast})$ be the figure eight. A crucial oberservation is that the first component $\gamma_1^{\ast}$ of $\gamma^*$ is already $2K(m^*)$-periodic as one could compute with the same  techniques as in \eqref{eq:periwave}. Note that
\begin{equation}
\gamma^*(K(m^*)) = \begin{pmatrix}
2 E(m^*)- K(m^*) \\  -2 \sqrt{m^*} \cos( \frac{\pi}{2} ) 
\end{pmatrix} = \begin{pmatrix}
0 \\ 0 
\end{pmatrix}
\end{equation} 
and by the $2K(m^*)$ periodicity of $\gamma_1^{\ast}$ we obtain
\begin{equation}
\gamma^{\ast}(3 K(m^*)) = \begin{pmatrix}
2 E(m^*)- K(m^*) \\  -2 \sqrt{m^*} \cos( 3\frac{\pi}{2} ) 
\end{pmatrix} = \begin{pmatrix}
0 \\ 0 
\end{pmatrix}.
\end{equation}
Hence $\gamma^{\ast}(K(m^*)) = \gamma(3 K(m^*))$. Next we show that the figure eight is minimizing. For this we compare the energy of the one-fold cover of the figure eight to the energy of the doubly-covered circle. 
The energy of the doubly covered circle $\gamma_{dcc}$ is given by 
\begin{equation}\label{eq:twocirc}
 \mathcal{E}(\gamma_{dcc}) \mathcal{L}(\gamma_{dcc}) = 4\pi \cdot 4\pi = 16\pi^2. 
\end{equation}
For the energy of the figure eight note that $\mathcal{L}(\gamma^*) = 4K(m^*)$ as $\gamma^*$ is arc-length parametrized by the construction in Appendix \ref{appendix:Jacobi}. Now 
\begin{align}
\mathcal{E}(\gamma^*) & = \int_0^{4K(m^*)} 4 m^* \cn^2(s,m^*) \ds = 4 m^* \int_0^{2\pi} \frac{\cos^2(\theta)}{\sqrt{1- m^* \sin^2(\theta) }} \;  \mathrm{d}\theta  \\ & = 
16 [(m^* - 1) K(m^*) + E(m^*)].
\end{align}
Hence  
\begin{equation}
\mathcal{E}(\gamma^*) \mathcal{L}(\gamma^*) = 64 \left[ E(m^*) K(m^*) + (m^*-1) K(m^*)^2\right].
\end{equation}

Using that by Definition \ref{def:figeight} $2E(m^*) = K(m^*)$, we have 
\begin{equation} \label{eq:ELOPT}
\mathcal{E}(\gamma^*) \mathcal{L}(\gamma^*) = 64 ( 4m^* -2) E(m^*)^2.
\end{equation}
We show that this quantity is smaller than $16 \pi^2$. For this we use that by Lemma \ref{lem:Heumanlambda} one has 
\begin{equation} \label{eq:vergleich}
64 ( 4m^* -2) E(m^*)^2  \leq 16 (2m^*-1) (2-m^*) \pi^2.
\end{equation}
With standard arguments it can be shown that $g(z) \defeq (2z-1)(2-z)$ is strictly monotone on $(0,1]$ and $g(1) = 1$ which makes $(2m^* - 1)(2-m^*) =  g(m^*) <1$. Therefore, by \eqref{eq:ELOPT} and \eqref{eq:vergleich}
\begin{equation}\label{eq:fig eight 16pi^2}
\mathcal{E}(\gamma^*) \mathcal{L}(\gamma^*) = 64 ( 4m^* -2) E(m^*)^2  < 16  \pi^2,
\end{equation}
 which implies by \eqref{eq:twocirc} that 
 
 \begin{align}
c^*=\mathcal{E}(\gamma^*) \mathcal{L}(\gamma^*) < \mathcal{E}(\gamma_{dcc}) \mathcal{L}(\gamma_{dcc}). &\qedhere
\end{align}

\end{proof}

\begin{remark}
	From \eqref{eq:fig eight 16pi^2} and \Cref{rem:value m*} one may compute $c^* \simeq 112.439609741$.
\end{remark}

\begin{lem}\label{lem:figureeight}
Let $\gamma^*$ be as in \Cref{def:figeight}. Then $S[\gamma^*]=\{0\}$  with $\mathrm{mult}[\gamma^*](0) = 2$. Moreover, $S_{tan}[\gamma^*] = \emptyset$.
\end{lem}
\begin{proof}
Note that the assertion is not affected by reparametrization. We will work with reparametrizations in the sequel. More exactly, we work with the parametrization $\tilde{\gamma}^{\ast}$ from \Cref{rem:3.2}, given by 
\begin{equation}
\tilde{\gamma}^{\ast} : [0,2\pi] \rightarrow \mathbb{R}^2 , \quad \begin{pmatrix}
\tilde{\gamma}^{\ast}_1(x) \\ \tilde{\gamma}^{\ast}_2(x) 
\end{pmatrix} =  \begin{pmatrix}
2 E(x, m^*) - F(x,m^*) \\ -2 \sqrt{m^*} \cos(x) 
\end{pmatrix}.
\end{equation} 
Now let $x_1, x_2 \in [0,2\pi)$ be such that $\tilde{\gamma}^{\ast}(x_1) = \tilde{\gamma}^{\ast}(x_2)$ and, without loss of generality, $x_1 < x_2$. Note that $\tilde{\gamma}^{\ast}_2(x_1) = \tilde{\gamma}^{\ast}_2(x_2)$ yields that $\cos(x_1) = \cos(x_2)$ which implies --- since $x_1, x_2 \in [0,2\pi)$ that $x_1 = 2\pi - x_2$. We infer that
\begin{align}
\tilde{\gamma}^{\ast}_1(x_1) & = \tilde{\gamma}^{\ast}_1(x_2) = \tilde{\gamma}^{\ast}_1(2\pi - x_1) = 2E(2\pi - x_1, m^*) - F(2\pi - x_1, m^*) 
\\ & = 2 E(-x_1, m^*) - F(-x_1,m^*) + (2 E(m^*) - K(m^*)) 
\\ & =  2 E(-x_1, m^*) - F(-x_1,m^*) = - ( 2 E(x_1,m) - F(x_1,m)) = -\tilde{\gamma}^{\ast}_1(x_1).
\end{align}
Hence $\tilde{\gamma}_1^{\ast}(x_1) = 0$ and therefore also $\tilde{\gamma}^{\ast}_1(x_2) = 0$. This means that $x_1$ and $x_2$ are solutions of 
\begin{equation}
2E(x,m^*)- F(x,m^*) = 0 
\end{equation}
By Lemma \ref{lem:e-Fzweo} this implies that 
\begin{equation}
x_1, x_2  \in \{ 0 , \tfrac{\pi}{2} , \pi , \tfrac{3\pi}{2} \} . 
\end{equation}
Since also $x_1 = 2\pi - x_2$ this leaves the only possibility of $x_1 = \frac{\pi}{2}, x_2 = \frac{3\pi}{2}$. We obtain that the only self intersection point occurs at $\tilde{\gamma}^{\ast}( \frac{\pi}{2} ) = \tilde{\gamma}^{\ast}( \frac{3\pi}{2} )  = (0,0)^T$. Consequently, $S[\tilde{\gamma}^{\ast}] = \{(0,0)^T\}$ with $\mathrm{mult}[\tilde{\gamma}^{\ast}]((0,0)^T) = 2$. It remains to show that $S_{tan}[\tilde{\gamma}^{\ast}] = \emptyset$, i.e. the self-intersection is not tangential. To do so we compute 
\begin{align}
\tilde{\gamma}^{\ast\prime}_1(x)& = \frac{1-2m^* \sin^2(x)}{\sqrt{1-m^* \sin^2(x)}}  & \tilde{\gamma}^{\ast\prime}_2(x) = 2 \sqrt{m^*} \sin(x) 
\end{align}
and therefore we have $\tilde{\gamma}^{\ast\prime}(\frac{\pi}{2}) = ( \frac{1-2m^*}{\sqrt{1-m^*}}, 2 \sqrt{m^*} )^T, \tilde{\gamma}^{\ast\prime}(\frac{3\pi}{2}) = ( \frac{1-2m^*}{\sqrt{1-m^*}}, - 2 \sqrt{m^*} )^T$ and thus 
\begin{equation}
\mathrm{det} ( \tilde{\gamma}^{\ast\prime}(\tfrac{\pi}{2}) ,  \tilde{\gamma}^{\ast\prime}(\tfrac{3\pi}{2})) = -4 \frac{(1-2m^*) \sqrt{m^*}}{\sqrt{1-m^*}} \neq 0 ,
\end{equation}
as $m^* \neq \frac{1}{2}$ by Lemma \ref{lem:uniqueroot}. 
\end{proof}

\subsection{Self-intersection properties of minimizers}

In this section, we will prove that every minimizer in \eqref{eq:infprob} is a constrained elastica and thus has to be the figure eight by \Cref{lem:charelas}, after rescaling, rotation, translation and reparametrization. This is a crucial and non-standard step in proving our main result. First, we examine the number of intersection points and their multiplicities. Then we show that minimizers have no tangential self-intersections. This allows us to conclude that all minimzers are interior points of $\CalA$, and hence elasticae. The key idea here is to compare the minimizer to the figure eight elastica defined in \Cref{def:figeight}.

Since self-intersection points are delicate to examine we will always localize.
We say that $\bar{\gamma}$ solves the Euler-Lagrange equation weakly on an open set $U \subset \mathbb{S}^1$ if \eqref{eq:EulerLag} holds for all $\phi \in C_0^\infty(U;\R^2)$.
\begin{lem}\label{lem:nonintersec} Let $\overline{\gamma} \in \mathcal{A}$ be a minimizer of \eqref{eq:infprob}. 
Suppose that $x \in \mathbb{S}^1$ is such that $\bar{\gamma}(x) \not \in S[\gamma]$. Then there exists an open neighborhood $U$ of $x$ in $\mathbb{S}^1$ such that for all $\phi \in C_0^\infty(U;\R^2)$ one has 
\begin{equation}\label{eq:EL1}
\CalL(\bar{\gamma})D\CalE(\bar{\gamma})(\phi) + \CalE(\bar{\gamma}) D\CalL(\bar{\gamma})(\phi) = 0.
\end{equation}
\end{lem}
\begin{proof}
Fix $x$ as in the statement. Since $\bar{\gamma} \in \mathcal{A}$ we have that $S[\bar{\gamma}] \neq \emptyset$. Hence there exists some $p \in S[\gamma]$ and two distinct values $x',x'' \in \mathbb{S}^1$ such that $\gamma(x') = \gamma(x'') = p$. Fix one choice of such $p,x',x''$. Now set  $U \defeq \mathbb{S}^1 \setminus \{x',x''\}$ which is open in $\mathbb{S}^1$  and contains $x$ as $\gamma(x) \neq p$.  Now fix $\phi \in C_0^\infty(U;\R^2)$.   We prove that $\phi$ satisfies \eqref{eq:EL1}. To do so we show that $\bar{\gamma} + \varepsilon \phi \in \mathcal{A}$ for all $\varepsilon \in \mathbb{R}$. As $x',x'' \not \in U$ we obtain that $\phi(x') = \phi(x'') = 0$ and thus for all $\varepsilon \in \mathbb{R}$ 
\begin{equation}
(\bar{\gamma} + \varepsilon \phi)(x') = \bar{\gamma}(x') = p = \bar{\gamma}(x'') = ( \bar{\gamma} + \varepsilon \phi)(x'') . 
\end{equation}
In particular $\bar{\gamma} + \varepsilon \phi \in \mathcal{A}$ as it has a self-intersection. Equation \eqref{eq:EL1} follows then from the consideration in \eqref{eq:variation}. 
\end{proof}

Once we have this tool at hand, we can start to study the self-intersections.
\begin{lem}\label{lem:singleton}
Let $\bar{\gamma} \in \mathcal{A}$ be a minimizer of \eqref{eq:infprob}. Then $S[\bar{\gamma}] = \{p_0\}$ for some $p_0 \in \mathbb{R}^2$.  
\end{lem}
\begin{proof}
Note first that $S[\bar{\gamma}] \neq \emptyset$ as $\bar{\gamma}$ is not embedded. We proceed showing that $S[\bar{\gamma}]$ is a singleton. Assume that there exist $p_1,p_2 \in S[\bar{\gamma}]$ such that $p_1 \neq p_2$. In particular there exists $\rho > 0$ such that $B_\rho(p_1) \cap B_\rho(p_2) = \emptyset.$ We claim that $\bar{\gamma}$ must be a constrained elastica. For this we show that the Euler-Lagrange equation is globally fulfilled. Thus, we have to discuss the behavior at self-intersection points. 
Fix any $x_0 \in \mathbb{S}^1$ such that $\bar\gamma(x_0) \in S[\bar\gamma]$. We will derive that $\bar\gamma$ solves the Euler-Lagrange equation in a neighborhood of $x_0$. We show first that there exists $\delta = \delta(x_0) > 0 $ such that $\bar\gamma ((x_0 - \delta, x_0 + \delta) )$ has empty intersection with \emph{one} of $B_{\frac{\rho}{2}}(p_1)$ or $B_{\frac{\rho}{2}}(p_2)$. For this we distinguish two cases, the first one being $\bar\gamma(x_0) \not \in B_\rho(p_1)$. By continuity there exists some $\delta> 0$ such that $|\bar\gamma(x) - \bar\gamma(x_0)| < \frac{\rho}{2}$ for all $x \in (x_0 - \delta, x_0 + \delta)$. Hence, by the triangle inequality, $\bar\gamma(x) \not \in B_{\frac{\rho}{2}}(p_1)$ for all $x \in (x_0- \delta , x_0 + \delta)$. The second case is $\bar\gamma(x_0) \in B_{\rho}(p_1)$, it could as well be equal to $p_1$. By the construction of $\rho$ this implies that $\bar\gamma(x_0) \not \in B_{\rho}(p_2)$ and thus we can repeat the above continuity argument to find that there exists $\delta > 0$ such that $\gamma(x) \not \in B_{\frac{\rho}{2}}(p_2)$ for all $x \in (x_0 - \delta, x_0 + \delta)$. With this case distinction we have completed the construction of $\delta = \delta(x_0)$. We will without loss of generality assume that $\bar{\gamma}((x_0 -\delta ,x_0 + \delta)) $ has empty intersection with $B_\frac{\rho}{2}(p_2)$, otherwise we switch roles.

% By continuity $\overline{\gamma}^{-1}(B_{\epsilon}(x_1))$ is open and in particular contains an open interval $(t_0- \delta, t_0 + \delta) \subset \mathbb{S}^1, \delta >0$ centered in $t_1$. 
 We show now that for all $\phi \in C_0^\infty((x_0 - \delta , x_0 + \delta);\R^2)$ one has 
\begin{equation}\label{eq:EL2}
\CalL(\bar{\gamma})D\CalE(\bar{\gamma})(\phi) + \CalE(\bar{\gamma}) D\CalL(\bar{\gamma})(\phi) = 0.
\end{equation}
To this end we fix $\phi \in C_0^\infty((x_0 - \delta , x_0 + \delta);\R^2)$. We show that for all $\varepsilon  \in \mathbb{R}$ one has $\bar{\gamma} + \varepsilon \phi \in \mathcal{A}$ which implies \eqref{eq:EL2} by minimality of $\bar{\gamma}$. Recall that $p_2 \in S[\bar\gamma]$ and hence there exist $x',x'' \in \mathbb{S}^1$ such that $\bar{\gamma}(x') = \bar{\gamma}(x'') = p_2$. We claim that $x',x'' \not \in (x_0 - \delta , x_0 + \delta)$. Indeed, if we assume e.g that $x' \in (x_0- \delta, x_0 + \delta)$, we obtain by choice of $\delta$ that $p_2 = \bar{\gamma}(x') \in \bar{\gamma}((x_0 - \delta, x_0+\delta))$. This is a contradiction to the fact that $\bar{\gamma}((x_0 - \delta , x_0 + \delta))$ does not intersect $B_\frac{\rho}{2}(p_2)$. Similarly one obtains that $x'' \not \in (x_0 - \delta, x_0 + \delta)$. Now we can compute for each $\varepsilon \in \mathbb{R}$ using that $\supp\phi \subset (x_0 - \delta , x_0 + \delta)$ 
\begin{equation}
(\bar{\gamma} + \varepsilon \phi)(x') = \bar{\gamma}(x') = p_2 = \bar{\gamma}(x'') = ( \bar{\gamma} + \varepsilon \phi) (x'') .  
\end{equation} 
This implies that $\bar{\gamma}+ \varepsilon \phi \in  \mathcal{A}$ and --- as we discussed --- also \eqref{eq:EL2}.

 Recalling that $x_0$ was arbitrary we have shown that for each $x_0 \in \mathbb{S}^1$ such that $x_0 \in S[\bar\gamma]$ there exists an open neighborhood $U_{x_0}$ of $x_0$ in $\mathbb{S}^1$ such that for all $\phi \in C_0^\infty(U_{x_0})$ equation \eqref{eq:EL2} holds true, under the assumption that there exist $p_1\neq p_2 \in S[\bar{\gamma}]$. Since such neighborhood exists also for non-intersection points by Lemma \ref{lem:nonintersec} we obtain that for each $x \in \mathbb{S}^1$ there exists an open neighborhood $U_x$ of $x$ in $\mathbb{S}^1$ such that the Euler-Lagrange equation \eqref{eq:EulerLag} holds for all $\phi \in C_0^\infty(U_x;\R^2)$. Since $\mathbb{S}^1$ is compact we may choose finitely many points $\{x_1,...,x_n\}$ such that $(U_{x_i})_{i = 1,...,n}$ is a cover of $\mathbb{S}^1$. Now we can choose a partition of unity of $\mathbb{S}^1$ subordinate to this cover, which yields non-negative functions $\eta_{x_1},...,\eta_{x_N} \in C^\infty(\mathbb{S}^1;\R)$ such that the support of $\eta_{x_i}$ is compactly contained in $U_{x_i}$ for all $i = 1,...,N$ and 
\begin{equation}
\sum_{i = 1}^N \eta_{x_i} = 1.
\end{equation} 
Now fix $\phi \in C^\infty(\mathbb{S}^1;\R^2)$. We obtain by \eqref{eq:EL2} and the linearity of the Frechét derivative
\begin{equation}
\CalL(\bar{\gamma})D\CalE(\bar{\gamma})(\phi) + \CalE(\bar{\gamma}) D\CalL(\bar{\gamma})(\phi) = \sum_{ i = 1}^N \CalL(\bar{\gamma})D\CalE(\bar{\gamma})( \eta_{x_i} \phi ) + \CalE(\bar{\gamma}) D\CalL(\bar{\gamma})(\eta_{x_i} \phi) = 0 .
\end{equation}
As discussed after \eqref{eq:EulerLag}, $\bar{\gamma}$ is a constrained elastica. Since $\bar{\gamma}$ minimizes \eqref{eq:infprob}, it must then also minimize among non-embedded constrained elasticae. However, the one-fold cover of the figure eight $\gamma^{\ast}$ is the unique minimizer of $\CalE\CalL$ among all constrained elasticae, up to rescaling, rotation and reparametrization, cf. Lemma \ref{lem:charelas}. Therefore, $\bar{\gamma}$ has to be a suitably rescaled, rotated and reparametrized version of $\gamma^{\ast}$. This yields that $S[\bar{\gamma}]$ is a singleton, since $S[\gamma^{\ast}]$ is a singleton, a contradiction.
\end{proof}

The ideas of the following proofs will be very similar to the preceding one. We assume that a certain configuration exists in $\bar{\gamma}$ and then conlude that $\bar{\gamma}$ has to be a constrained elastica. We then use the classification of those in Section \ref{sec:31} to rule out this configuration. The next lemma is in the same spirit.
\begin{lem}\label{lem:35}
Let $\bar{\gamma} \in \mathcal{A}$ be a minimizer of \eqref{eq:infprob}. Suppose that $p_0 \in S[\gamma]$. Then $\mathrm{mult}[\bar{\gamma}](p_0) = 2$. In particular there exist exactly two values $x_1,x_2 \in \mathbb{S}^1$ such that $x_1 \neq x_2$ and $\gamma(x_1) = \gamma(x_2)$. 
\end{lem}
\begin{proof}
As $S[\bar{\gamma}]$ is a singleton by Lemma \ref{lem:singleton} the last sentence of the claim follows immediately from the multiplicity result. As $p_0 \in S[\bar{\gamma}]$ {we have} $\mathrm{mult}[\bar{\gamma}](p_0) \geq 2$. Assume that $\mathrm{mult}[\bar{\gamma}](p_0) >2$. We show that then $\bar{\gamma}$ must be a constrained elastica. To this end, let $x_0 \in \gamma^{-1}(\{p_0\})$ be arbitrary. We claim that there exists a neighborhood $U_{x_0}$ on which the Euler Lagrange equation is fulfilled. By the assumption on the multiplicity there exist two distinct values $x', x'' \in \gamma^{-1}(\{p_0\}) \setminus\{x_0\}$. Fix a choice of such $x',x''$.  Choose $\delta =\delta(x_0) > 0$ such that $x', x'' \not \in (x_0- \delta , x_0+ \delta)$. We claim that then  for all $\phi \in C_0^\infty((x_0- \delta, x_0 + \delta);\R^2)$ and $\varepsilon \in \mathbb{R}$ one has $\bar{\gamma} + \varepsilon \phi  \in \mathcal{A}$. This is {true} since $\phi(x') = \phi(x'') = 0$ and thus
\begin{equation}
(\bar{\gamma} + \varepsilon \phi)(x') = \bar{\gamma}(x') = p_0 = \bar{\gamma}(x'') = ( \bar{\gamma} + \varepsilon \phi) (x'') .   
\end{equation}
{As a result we may conclude}
\begin{equation}
\CalL(\bar{\gamma})D\CalE(\bar{\gamma})(\phi) + \CalE(\bar{\gamma}) D\CalL(\bar{\gamma})(\phi) = 0 \quad \forall \phi \in C_0^\infty((x_0- \delta, x_0+\delta);\R^2).
\end{equation}
We have shown that for all $x_0 \in \gamma^{-1}(\{p_0\})$ there exists a neighborhood $U_{x_0}$ such that $\bar{\gamma}$ solves the Euler Lagrange equation weakly on $U_{x_0}$. {By \Cref{lem:singleton}, we have} $S[\bar{\gamma}] =\{ p_0 \}$ and hence we infer that for all $p_0 = \gamma(x_0) \in S[\bar{\gamma}]$ there exists a neighborhood $U_{x_0}$ such that $\bar{\gamma}$ solves the Euler Lagrange equation weakly on $U_{x_0}$. Together with Lemma \ref{lem:nonintersec} we conclude that each $x\in  \mathbb{S}^1$ has an open neighborhood $U_x$ such that $\bar{\gamma}$ solves \eqref{eq:EulerLag} weakly on $U_x$. One can repeat the partition of unity argument in the proof of Lemma \ref{lem:singleton} to find that $\bar{\gamma}$ solves \eqref{eq:EulerLag} globally  in $\mathbb{S}^1$. Hence $\bar{\gamma}$ is by the discussion after \eqref{eq:EulerLag} a constrained elastica. By minimality of $\bar{\gamma}$, it follows that $\bar{\gamma}$ must minimize $\mathcal{E}\mathcal{L}$ also among non-embedded elasticae and hence is a rescaled, rotated and translated reparametrization of the one-fold covered figure eight $\gamma^*$, cf. Lemma \ref{lem:charelas}. As $\mathrm{mult}[\gamma^*](p) \leq 2$ for all $p \in S[\gamma^*]$ we infer that $\mathrm{mult}[\bar{\gamma}](p_0) \leq 2${, a contradiction.}
%. The claim follows.
\end{proof}

Next we show that $\bar{\gamma}$ has no tangential self-intersections. As a preparation for this, we discuss an important quantity, the \emph{winding number} $T$, which is defined in Definition \ref{def:winding} in Appendix \ref{sec:AppA}.
\begin{prop}\label{prop:WindungszahlInner}
Let $\gamma \in \mathcal{A}$ such that $T[\gamma] \neq \pm 1$. Then $\gamma$ is an interior point of $\mathcal{A}$ with respect to the $W^{2,2}(\S^1;\R^2)$-norm.
\end{prop} 
\begin{proof} 
	Let $\gamma\in \CalA$. Since $T\colon W_{Imm}^{2,2}(\S^1;\R^2)\to \R$ is continuous and $\Z$-valued, it is locally constant near $\gamma$. Thus, if $T[\gamma]\neq \pm 1$, we have $T[\tilde{\gamma}]= T[\gamma]\neq \pm1$ for all $\tilde{\gamma}\in W^{2,2}(\S^1;\R^2)$ with $\Norm{\tilde{\gamma}-\gamma}{W^{2,2}}<\delta$ for $\delta>0$ small enough. Consequently, by Hopf's Umlaufsatz, \Cref{thm:HopfW22}, any such $\tilde{\gamma}$ is not an embedding, hence $\tilde{\gamma}\in \CalA$.
\end{proof}

The winding number can now be used to detect interior points of $\CalA$, which we will use next to exclude tangential self-intersections.

\begin{lem}\label{lem:NoTangetialSelfIntersections}
Let $\bar{\gamma} \in \mathcal{A}$ be a minimizer in \eqref{eq:infprob}. Then $S_{tan}[\bar{\gamma}] = \emptyset$. 
\end{lem}
\begin{proof}
    	By \Cref{lem:35}, we have $S[\bar{\gamma}]=\{p_0\}$ with $\mathrm{mult}[\bar{\gamma}](p_0)=2$. 
	Assume that $p_0$ is a tangential self-intersection with multiplicity two and $\bar{\gamma}^{-1}(\{p_0\})=\{x_0, x_1\}$. After reparametrization, we may assume that $\bar{\gamma}$ is parametrized with constant speed, i.e. $\abs{\bar\gamma^{\prime}(x)}=\CalL(\bar\gamma)$ for all $x\in \S^1$. Hence, we have $\bar{\gamma}^{\prime}(x_0) = \pm \bar{\gamma}^{\prime}(x_1)$. \newline
	We first consider the case $\bar{\gamma}^{\prime}(x_0)=\bar{\gamma}^{\prime}(x_1)$. Then $\bar{\gamma}_1\defeq \bar{\gamma}\vert_{[x_0, x_1]}$ and $\bar{\gamma}_2	\defeq \bar{\gamma}\vert_{[x_1, x_0]}$ are $C^1$-closed embedded  %${C}^{1}$-simple closed
	 curves and after an appropriate reparametrization, we have $\bar{\gamma}_1, \bar{\gamma}_2\in W^{2,2}(\S^1;\R^2)$ using \Cref{rem:13}. We obtain
	\begin{align}
		T[\bar{\gamma}] = T[\bar{\gamma}_1]+T[\bar{\gamma}_2]. 
	\end{align}
	Since $\bar{\gamma}_j$ is simple closed, we obtain $T[\bar{\gamma}_j] = \pm 1$ for $j=1,2$. Hence $T[\bar{\gamma}]\neq \pm 1$, so $\bar{\gamma}$ is an interior point of $\CalA$ by  \Cref{prop:WindungszahlInner}. Consequently $\bar{\gamma}$ satisfies \eqref{eq:variation} and thus \eqref{eq:EulerLag}. Consequently, $\bar{\gamma}$ is an elastica by the discussion after \eqref{eq:EulerLag}. Therefore, by minimality, must be a rescaled translated and rotated reparametrization of $\gamma^*$ by Lemma \ref{lem:charelas}. This contradicts Lemma \ref{lem:figureeight}.
	
	For the case $\bar{\gamma}'(x_0) =- \bar{\gamma}'(x_1)$ we can without loss of generality assume that $x_0 =0$ and $\bar{\gamma}'(0) = (\CalL(\bar{\gamma}),0)^T$. Let $\theta \in W^{1,2}((0,1); \mathbb{R})$ be the angle function from \Cref{lem:liftingRegularity} with $\theta(0) = 0$. Then $\CalL(\bar\gamma)\kappa(x) = \theta'(x)$ by \eqref{eq:theta_kappa}. Suppose now that $\bar{\gamma}$ is not an interior point of $\mathcal{A}$. Hence $T[\bar\gamma] = \pm 1$ by \Cref{prop:WindungszahlInner}. After appropriate reparametrization we may assume that $T[\bar\gamma] = 1$ and hence 
	\begin{equation}
	    1 = \frac{1}{2\pi} \int_\gamma \kappa \; \mathrm{d}s = \frac{\theta(1) - \theta(0)}{2\pi}.
	\end{equation}
	In particular $\theta(1)= 2\pi$. As $\theta(0) = 0 $ and $\bar{\gamma}'(x_1) = - \bar{\gamma}'(0) = (-\CalL(\bar{\gamma}),0)^T$ we infer that $\theta(x_1) = k \pi$ for some odd number $k \in 2 \mathbb{Z} + 1$. Now we define 
	\begin{equation}
	\widetilde{\gamma}(x) \defeq  \begin{cases}
	\bar{\gamma}(x) & x \in [0, x_1] \\ \bar{\gamma}(1- (x-x_1)) & x \in [x_1 , 1].
	\end{cases}
	\end{equation}
	The curve $\widetilde{\gamma}$ is well-defined since $\bar{\gamma}(0) = \bar{\gamma}(x_1) = \bar{\gamma}(1)$ using \Cref{rem:13}. Note in particular that $x_1 \neq 0$  in $\mathbb{S}^1$
	and hence $ \widetilde{\gamma}$ is not injective. We claim that $\widetilde{\gamma}$ is another minimizer. 
	To this end, we show that $\widetilde{\gamma} \in W^{2,2}(\mathbb{S}^1; \mathbb{R}^2)$  and $\mathcal{E}(\widetilde{\gamma}) \mathcal{L}(\widetilde{\gamma})= \mathcal{E}(\bar{\gamma}) \mathcal{L}(\bar{\gamma}).$ By Remark \ref{rem:13} it suffices to show that the zeroth and first derivatives of the two cases coincide at $x=0=1$ and at $x=x_1$. This is easy to check using that $\bar{\gamma}'(0) = \bar{\gamma}'(1) = - \bar{\gamma}'(x_1)$.  It is also immediate to check that $\mathcal{E}(\widetilde{\gamma}) = \mathcal{E}( \bar{\gamma})$ and $\mathcal{L}(\widetilde{\gamma}) = \mathcal{L}(\bar{\gamma})$.
	 Hence $\widetilde{\gamma}$ is another minimizer as claimed. Observe also that $S_{tan}[\widetilde{\gamma}] = S_{tan}[\bar{\gamma}] \neq \emptyset$ since $\widetilde{\gamma}(0) = \widetilde{\gamma}(x_1)$ and  $\widetilde{\gamma}'(0) = - \widetilde{\gamma}'(x_1)$. 
	We now claim that
	\begin{equation}
	 \widetilde{\gamma}'(x) = \mathcal{L}(\widetilde{\gamma}) \begin{pmatrix}
	 \cos(\widetilde{\theta}(x)) \\ \sin( \widetilde{\theta}(x) ) 
	 \end{pmatrix} \quad \forall x \in (0,1),
	\end{equation}
	where $\tilde{\theta}$ is given by
	\begin{equation}\label{eq:thetatilde}
	\widetilde{\theta}(x) = \begin{cases}
	\theta(x) & x \in (0, x_1] \\ (k-2) \pi + \theta(1- (x-x_1)) & x \in (x_1, 1),
	\end{cases} 
	\end{equation}
	 with $k \in 2 \mathbb{Z} + 1$ as before.  To show that $\widetilde{\theta}$ is the angle function of $\widetilde{\gamma}$, we observe that for $x \in (x_1,1)$ 
	\begin{align}
	\widetilde{\gamma}'(x) = - \bar{\gamma}'(1- (x-x_1) ) &= - \mathcal{L}(\bar{\gamma}) \begin{pmatrix}
	\cos(\theta(1- (x-x_1))) \\ \sin(\theta(1-(x-x_1))) 
	\end{pmatrix} \\
	&=\mathcal{L}(\bar{\gamma}) \begin{pmatrix}
	\cos( \pi + \theta(1- (x-x_1))) \\ \sin( \pi + \theta(1-(x-x_1))) 
	\end{pmatrix}.
	\end{align}
	{Hence, by \Cref{lem:liftingRegularity}, $\widetilde{\theta}$ and $\pi + \theta( 1- (\cdot - x_1)) $ differ only by a constant multiple of $2\pi$.  The multiple has to be chosen in such a way that $\widetilde{\theta} \in W^{1,2}(0,1)$. Since $\theta(x_1) = k\pi = (k-2) \pi + \theta(1)$, the only possible choice for $\widetilde{\theta}$ is the one we defined in \eqref{eq:thetatilde}.} Using this and the relation $\kappa\abs{\gamma^{\prime}}=\theta^{\prime}$ (cf. \eqref{eq:theta_kappa}) we compute %Note that we had no choice to choose this multiple since we needed that $\widetilde{\theta}$
	%Recall now that $\theta(x_1) =  k \pi$ for some odd integer $k$ and compute   
	\begin{align}
	\frac{1}{2\pi} \int_{\S^1} \kappa[\tilde{\gamma}] \; \mathrm{d}s_{\tilde{\gamma}} & = \frac{1}{2\pi} \left( \int_0^{x_1} \theta'(x) \; \mathrm{d}x - \int_{x_1}^1 \theta'(1-(x-x_1)) \; \mathrm{d}x \right) \\ & = \frac{1}{2\pi} \left( \int_0^{x_1} \theta'(y) \; \mathrm{d}y - \int_{x_1}^{1} \theta'(y) \; \mathrm{d}y \right) \\ &  = \frac{1}{2\pi} ( 2  \theta(x_1) - \theta(0) - \theta(1) )  = k-1,
	\end{align}
	since $\theta(0) = 0, \theta(1) = 2\pi$. In particular by \Cref{thm:WindinIntegral}
	 \begin{equation}
	 T[\widetilde{\gamma}] \not \in \{ + 1, -1\}
	 \end{equation}
	 since $k-1$ is even as $k$ is odd. We infer that $\widetilde{\gamma}$ is an interior point of $\mathcal{A}$. Since it is also a minimizer, it must by the same arguments as in the beginning of the proof be a rescaled, translated and reparametrized version of $\gamma^*$. However, this is a contradiction, since $S_{tan}[\gamma^*]=\emptyset$ by \Cref{lem:figureeight} but $S_{tan}[\widetilde{\gamma}] = S_{tan}[\bar{\gamma}] \neq \emptyset$. 
\end{proof}

Once we can rule out tangential self-intersections (as in \Cref{fig:tangentialSelfIntersection}), we can finally show that a minimizer is an interior point.

The following lemma shows that non-tangential self-intersections are stable under $C^1$-small perturbations.

	\begin{lem}\label{lem:nontang implicit function}
		Let $\bar{\gamma}\in C^1(\S^1;\R^2)$ and assume $\bar{\gamma}$ has a single non-tangential self intersection with multiplicity two, i.e. $S[\bar{\gamma}] =\{p\}$ with $\bar{\gamma}^{-1}(\{p\})=\{\bar{x}_1, \bar{x}_2\}$ for $\bar{x}_1\neq \bar{x}_2$ and $S_{tan}[\bar{\gamma}]=\emptyset$. Then, there exists $\delta>0, \varepsilon_0>0$ such that
		$(\bar{x}_1-\delta, \bar{x}_1+\delta)\cap (\bar{x}_2-\delta, \bar{x}_2+\delta)=\emptyset$ and	
		every curve $\gamma\in  B\defeq \{ \eta\in C^1(\S^1;\R^2) \mid \norm{\gamma-\bar{\gamma}}{C^1}< \varepsilon_0\}$ has a unique self intersection, i.e. there exist unique $x_1\neq x_2\in \S^1$ such that $\gamma(x_1)=\gamma(x_2)$. Moreover, this self-intersection is non-tangential, satisfies $x_i\in [\bar{x}_i-\delta, \bar{x}_i+\delta]$ for $i=1,2$ and the function $B\ni \gamma \mapsto (x_1, x_2)\in \R^2$ is of class $C^1$.
	\end{lem}
	\begin{proof}
		Let $\delta>0$ be small enough such that $(\bar{x}_1-\delta, \bar{x}_1+\delta)\cap (\bar{x}_2-\delta, \bar{x}_2+\delta)=\emptyset$. Moreover, taking $\varepsilon_0>0$ small enough and using \Cref{lem:EmbeddChara} and \Cref{rem:embeddings ab}, we may assume that any $\gamma \in B$ is immersed and injective when restricted to	
		$\S^1\setminus (\bar{x}_i-\delta, \bar{x}_i+\delta)$ for $i=1,2$. Now, we define $\mathcal{U}\defeq (\bar{x}_1-\delta, \bar{x}_1+\delta)\times (\bar{x}_2-\delta, \bar{x}_2+\delta) \times B$ and the function
		\begin{align}
			\Phi: \mathcal{U}  \to \R^2, \quad 
			\Phi(x_1,x_2, \gamma) \defeq\gamma(x_1)-\gamma(x_2).
		\end{align}
		Then $\Phi(\bar{x}_1, \bar{x}_2, \bar{\gamma})=0$ by assumption. Moreover, $\Phi$ is of class $C^1$ since the map $\varphi\colon(a,b)\times B\to \R^2, (x,\gamma)\mapsto \gamma(x)$ is $C^1$ for any $a<b$ with derivative $D\varphi(x,\gamma)[z, \eta] = \eta(x)+\gamma'(x)z$ for $x\in (a,b), z\in \R$, $\gamma\in B$ and $\eta\in C^1(\S^1;\R^2)$. Now, the partial derivative $D_{(x_1,x_2)}\Phi(\bar{x}_1, \bar{x}_2, \bar{\gamma})\colon \R^2\to \R^2$ given by
		\begin{align}
			D_{(x_1,x_2)}\Phi(\bar{x}_1, \bar{x}_2, \bar{\gamma})[z_1, z_2] = \bar{\gamma}'(\bar{x}_1)z_1 - \bar{\gamma}'(\bar{x}_2)z_2\quad \text{for }z_1, z_2\in\R.
		\end{align} 
		is invertible, since $\bar{\gamma}'(\bar{x}_1)$ and $\bar{\gamma}'(\bar{x}_2)$ are linearly independent. Since $\Phi$ is $C^1$, we may hence assume that  
		\begin{align}
			D_{(x_1, x_2)}\Phi(x_1, x_2, \gamma) \text{ is invertible for all }  (x_1, x_2, \gamma) \in \mathcal{U}.\label{eq:DPhi invertible}
		\end{align}
		By the implicit function theorem \cite[Theorem 4.B]{Zeidler}, after possibly reducing $\varepsilon_0>0$ and $\delta>0$, for all $\gamma \in B$ there exist unique $x_i=x_i(\gamma) \in (\bar{x}_i-\delta, \bar{x}_i+\delta), i=1,2,$ with $\Phi(x_1(\gamma), x_2(\gamma), \gamma)=0$ for all $\gamma\in B$. Moreover, the map $\gamma\mapsto (x_1(\gamma),x_2(\gamma))$ is of class $C^1$. The uniqueness of the self-intersection follows from the local uniqueness and the fact that $\gamma$ is injective on $\S^1\setminus  (\bar{x}_i-\delta, \bar{x}_i+\delta)$ for both $i=1$ and $i=2$. Furthermore, by \eqref{eq:DPhi invertible} this self-intersection is always non-tangential.
	\end{proof}

Equipped with this result, we can show that our minimizer satisfies the Euler--Lagrange equation.
\begin{lem} \label{lem:minielass}Let $\bar{\gamma}\in \mathcal{A}$ be a minimizer in \eqref{eq:infprob}. Then $\bar{\gamma}$ is a constrained elastica. 
\end{lem}

	\begin{proof}
		By \Cref{lem:35} and \Cref{lem:NoTangetialSelfIntersections}, there exist exactly two values $x_1\neq x_2\in \S^1$ with $\bar{\gamma}(x_1) = \bar{\gamma}(x_2)=p_0$ such that $\bar{\gamma}^{\prime}(x_1)$ and $\bar{\gamma}^{\prime}(x_2)$ are linearly independent. Let $\phi\in C^{\infty}_0(\S^1;\R^2)$. By \Cref{lem:nontang implicit function}, there exists $\varepsilon_0>0$ such that for all $\varepsilon\in (-\varepsilon_0, \varepsilon_0)$ the curve $\bar{\gamma}+\varepsilon\phi$ has a self-intersection, so $\bar{\gamma}+\varepsilon\phi\in \mathcal{A}$. Then, $(-\varepsilon_0, \varepsilon_0)\ni \varepsilon\mapsto \CalE(\bar{\gamma}+\varepsilon\phi)\CalL(\bar{\gamma}+\varepsilon\phi)$ has a local minimum in $\varepsilon=0$. We conclude
		\begin{align}
			\CalL(\bar{\gamma})D\CalE(\bar{\gamma})(\phi) + \CalE(\bar{\gamma}) D\CalL(\bar{\gamma})(\phi) = 0 \quad \forall \phi \in C_0^\infty(\S^1;\R^2). &\qedhere 
		\end{align}  
	\end{proof}

Finally, we can prove our main result.
\begin{proof}[Proof of Theorem \ref{thm:main}]
By \Cref{lem:variational} there exists $\bar{\gamma} \in \mathcal{A}$ such that $$\mathcal{E}(\bar{\gamma}) \mathcal{L}(\bar{\gamma})= \inf_{ \gamma \in \mathcal{A}} \mathcal{E}(\gamma) \mathcal{L}(\gamma).$$ 
By Lemma \ref{lem:minielass} we infer that $\bar{\gamma} \in \mathcal{B}$, where $\mathcal{B}$ is defined as in \eqref{eq:defB}. This implies together with Lemma \ref{lem:charelas} that
\begin{equation}
\inf_{\gamma \in \mathcal{A}} \mathcal{E}(\gamma) \mathcal{L}(\gamma) = \mathcal{E}(\bar{\gamma}) \mathcal{L}(\bar{\gamma}) = \inf_{\gamma \in \mathcal{B}} \mathcal{E}(\gamma) \mathcal{L}(\gamma) =  \mathcal{E}(\gamma^*) \mathcal{L}(\gamma^*). 
\end{equation}
The claim follows.
\end{proof}
\begin{remark}\label{rem:5.12}
If we seek to generalize Theorem \ref{thm:main} in higher codimension, e.g. in $\mathbb{R}^3$, the arguments in this section do not immediately carry over as more elasticae would need to be discussed in a generalized version of Lemma \ref{lem:charelas}.  %So far we do not think that an exhaustive classification of (non-embedded) elastic curves in $\mathbb{R}^3$ is known. 
\\
An exhaustive classification of elasticae in $\mathbb{R}^3$ in \cite{LangerKnot} shows that all non-planar elasticae $\gamma$ are embedded and knotted. Therefore, by the Fáry--Milnor Theorem (cf. \cite{Fary,MilnorKnots}) their energy is bounded from below since
\begin{equation}
\mathcal{E}(\gamma) \mathcal{L}(\gamma) =  \int_{\S^1} |\kappa|^2 \; \mathrm{d}s \int_{\S^1} 1 \; \mathrm{d}s \geq  \left( \int_{\S^1} |\kappa| \; \mathrm{d}s \right)^2 \geq 16 \pi^2>c^*,
\end{equation}
where we used \eqref{eq:fig eight 16pi^2} in the last inequality.
Clearly, multi-fold covers of these non-planar elasticae will only have higher energy. Hence, $\gamma^*$ is still the minimizer among all non-embedded elasticae in $\R^3$, generalizing \Cref{lem:charelas} to higher codimension.
%
%
% where all nonplanar closed elasticae are found.
% Let us call those elasticae $LS$-elasticae. It is shown that all $LS$-elasticae are embedded and knotted.
% For the minimization of $\mathcal{E}\mathcal{L}$ among non-embedded elasticae we thus have to regard non-embedded planar elasticae and multi-fold covers of $LS$-elasticae. A minimizer among non-embedded planar elasticae is again $\gamma^*$. We claim that the energy of all multi-fold covers of $LS$-elasticae is larger than the energy of $\gamma^*$. 
%Indeed,  by Fary-Milnor's Theorem (cf. \cite{MilnorKnots}) we obtain for each curve $\gamma$ that is a $k$-fold $(k \geq 2)$ cover of an LS-elastica $\gamma_0$ we have
%\begin{equation}
%\mathcal{E}(\gamma) \mathcal{L}(\gamma) \geq k^2 \mathcal{E}(\gamma_0) \mathcal{L}(\gamma_0) = k^2 \int_\gamma |\kappa|^2 \; \mathrm{d}s \int 1 \; \mathrm{d}s \geq k^2 \left( \int_\gamma |\kappa| \; \mathrm{d}s \right)^2 \geq 16 k^2 \pi^2 \ge 64 \pi^2. 
%\end{equation}
%This exceeds $\mathcal{E}(\gamma^*) \mathcal{L}(\gamma^*)= 112.43 < 16 \pi^2$ by a lot. \\ 
%Having generalized Lemma \ref{lem:charelas}
Moreover, we can prove as in Lemma \ref{lem:singleton} and  Lemma \ref{lem:35} that there exists a minimizer $\bar{\gamma}$ of $\mathcal{E} \mathcal{L}$ among closed non-embedded curves in $\mathbb{R}^3$ which has only one point of self-intersection with multiplicity $2$. However, the methods used in \Cref{lem:NoTangetialSelfIntersections} and \Cref{lem:minielass} are only available in codimension one, making it unclear, whether the minimizer is an elastica. This is the only obstruction to a generalization of \Cref{thm:main} to higher codimension.
%Let us summarize what we can say in $\mathbb{R}^3$. It should be possible to show with the arguments before that there exists a non-embedded minimizer $\bar{\gamma} \in W^{2,2}(\mathbb{S}^1, \mathbb{R}^3)$ of $  \mathcal{E}(\gamma) \mathcal{L}(\gamma)$ among all immersed curves $\gamma \in W^{2,2}( \mathbb{S}^1, \mathbb{R}^3)$ that are non-embedded.
% we therefore need to understand the non-embedded closed constrained elasticae in $\mathbb{R}^3$. Those can be divided into planar elasticae and nonplanar elasticae. The planar elasticae are the same ones that we already discussed here. By  \cite[Main Theorem]{LangerKnot}, all (one-fold covers) of nonplanar elastica are embedded. Hence the only new elastica that need to be examined are multi-fold covers of nonplanar closed elasticae. We expect that all of these have a higher energy than $\gamma^*$.}
\end{remark}

\section{An application: the elastic flow}\label{sec:flow}

We consider a family of smooth curves $\gamma\colon [0,T)\times\S^1\to\R^2$ evolving with respect to the gradient flow equation
\begin{align}\label{eq:Elastic FLow}
	\partial_t \gamma & = - \nabla_s^2 \vKap -\frac{1}{2} \abs{\vKap}^2\vKap + \lambda\vKap,
\end{align}
where $\vKap=\vKap[\gamma] = \partial_s^2 \gamma$ is the curvature vector and $\nabla_s = \partial_s^{\perp_{\gamma}}$ is the normal part of the arc-length derivative. 
Here, we either consider the \emph{length penalized} elastic flow, where $\lambda\geq 0$ is a fixed number  or the \emph{length preserving} elastic flow, where $\lambda=\lambda(\gamma(t,\cdot))\in\R$ depends on the solution and is given by
\begin{align}\label{eq:defLambda}
	\lambda &=	\frac{\int_{\S^1} \langle\nabla_s^2 \vKap + \frac{1}{2}\abs{\vKap}^2\vKap, \vKap\rangle \diff s}{\int_{\S^1} \abs{\vKap}^2\diff s}.
\end{align}
It can be easily checked that the length remains constant along solutions of \eqref{eq:Elastic FLow} with $\lambda$ given by \eqref{eq:defLambda}, since

\begin{align}\label{eq:length constant}
	\frac{\diff}{\diff t}\CalL(\gamma) = \int_{\S^1} \langle\nabla\CalL(\gamma), \partial_t\gamma\rangle \diff s = -\int_{\S^1} \langle \vKap, \partial_t \gamma\rangle\diff s=0,
\end{align}

by \eqref{eq:defLambda}.
Both geometric flows have been studied in \cite{DKS}, where long-time existence and subconvergence as $t\to\infty$ has been established. Using \Cref{thm:main}, we establish an energy bound which guarantees embeddedness along the flow. Note that in contrast to second-order evolutions, this does not follow from a maximum principle, as \eqref{eq:Elastic FLow} is of fourth order. Moreover as in \cite{Loja,MantegazzaPozzetta,RuppSpener} we can apply a suitable \emph{{\L}ojasiewicz--Simon inequality} to deduce convergence and then \Cref{lem:charelas} to give a precise characterization of the limit.

\begin{thm}\label{thm:conv length pres}
	Let $\gamma_0\in C^{\infty}(\S^1;\R^2)$ be an embedded curve such that $\CalE(\gamma_0)\CalL(\gamma_0)<c^*$. Then the elastic flow \eqref{eq:Elastic FLow} with fixed length and initial datum $\gamma(0,\cdot)=\gamma_0$ remains embedded for all times and converges, as $t\to\infty$, after reparametrization with constant speed to a one-fold cover of a circle with radius $\frac{\CalL(\gamma_0)}{2\pi}$.
\end{thm}

\begin{proof}
	We write $\gamma(t)\defeq \gamma(t, \cdot)$ and observe that the energy $\CalE(\gamma(t))$ is decreasing, while $\CalL(\gamma(t))=\CalL(\gamma_0)$ is kept fixed. Hence for $t\geq  0$, we have $\CalE(\gamma(t))\CalL(\gamma(t))<\CalE(\gamma^{\ast})\CalL(\gamma^{\ast})$ by the assumption on the initial datum, so $\gamma(t)$ is embedded for all $t\geq 0$ by \Cref{thm:main}. Now, by \cite[Theorem 3.3]{DKS}, the flow exists for all times and it holds 	$\Norm{\partial_s^{m}\vKap(t)}{L^{\infty}} \leq C_m,$
	for some $C_m>0$ and all $m\in \N_0$, $t\in [0,\infty)$. Thus, if $\tilde{\gamma}$ denotes the reparametrization by arc-length, we get
	\begin{align}\label{eq:convbound 0}
	\norm{\partial^m_x\tilde{\gamma}(t,\cdot)}{L^{\infty}}\leq C_m \text{ for all }m\in\N, t\in [0,\infty).
	\end{align}
	If we define the integral average $p(t)\defeq \int_{\S^1} \tilde{\gamma}(t, \cdot)\diff x\in\R^2$, we find
	\begin{align}\label{eq:convbound 1}
	\norm{\tilde{\gamma}(t, \cdot)-p(t)}{L^{\infty}} \leq \CalL(\tilde{\gamma}(t)) = \CalL(\gamma_0).
	\end{align}
	Now, if $t_n\to\infty$ is any sequence, using the Arzelà--Ascoli theorem and a diagonal sequence argument, after passing to a subsequence, we find $\tilde{\gamma}(t_n)-p(t_n)\to \gamma_{\infty}$ in $C^m(\S^1;\R^2)$ as $n\to\infty$ for every $m\in \N_0$, where $\gamma_{\infty}\in C^{\infty}(\S^1;\R^2)$ is a closed elastica, cf. \cite[Theorem 3.2]{DKS}. Since $\CalE(\gamma_\infty)\CalL(\gamma_\infty)<\CalE(\gamma^{\ast})\CalL(\gamma^{\ast})$, the curve $\gamma_\infty$ is embedded and has to be a (translation and reparametrization) of the one-fold cover of a circle by \Cref{lem:charelas}. Consequently, its radius has to be $\frac{\CalL(\gamma_0)}{2\pi}$ since $\CalL(\gamma_\infty)=\CalL(\gamma_0)$. However, different sequences could still yield circles with different centers.
	
	Even so, this cannot happen, since $\CalE$ satisfies a constrained {\L}ojasiewicz--Simon gradient inequality, cf. \cite{RuppLoja}. This can be proven using \cite[Corollary 5.2]{RuppLoja}, since the energies $\CalE$ and $\CalL$ are analytic and the length is of lower order, see also {\cite[Theorem 4.8]{RuppSpener}} for the analogous argument in the case of clamped curves. Hence, there exist constants $C_{LS}, \sigma>0$ and $\theta\in (0, \frac{1}{2}]$ such that for all $\gamma\in W^{4,2}(\S^1;\R^2)$ with $\norm{\gamma-\gamma_{\infty}}{W^{4,2}}\leq \sigma$ and $\CalL(\gamma)=\CalL(\gamma_0)$ we have

	\begin{align}\label{eq:Lojasiewicz}
	\abs{\CalE(\gamma)-\CalE(\gamma_{\infty})}^{1-\theta}\leq C_{LS}\norm{\nabla\CalE(\gamma)+\lambda(\gamma)\nabla\CalL(\gamma)}{L^2},
	\end{align}
	
	with $\lambda(\gamma)$ as in \eqref{eq:defLambda}. To prove the full convergence statement, we furthermore assume $\norm{\tilde{\gamma}(t_n)-p(t_n)-\gamma_\infty}{W^{4,2}}< \sigma$ for all $n\in \N$. We define
	\begin{align}
	s_n \defeq \sup\{s\geq t_n \mid \norm{\tilde{\gamma}(t)-p(t)-\gamma_\infty}{W^{4,2}}<\sigma\text{ for all }t\in [t_n, s]\},
	\end{align}
	and observe $s_n>t_n$ by smoothness. Now, the function
	\begin{align}
	G(t)\defeq \left(\CalE(\tilde{\gamma}(t))-\CalE(\gamma_\infty)\right)^{\theta} = \left(\CalE({\gamma}(t))-\CalE(\gamma_\infty)\right)^{\theta}
	\end{align}
	is decreasing and satisfies $\lim_{t\to\infty}G(t)=0$. Moreover, since $\gamma$ solves \eqref{eq:Elastic FLow} we have
	\begin{align}
	-\frac{\diff}{\diff t}G &= \theta \left(\CalE(\tilde{\gamma})-\CalE(\gamma_\infty)\right)^{\theta-1} \left(-\frac{\diff}{\diff t}\CalE({\gamma})\right) \\
	&= -\theta \left(\CalE(\tilde{\gamma})-\CalE(\gamma
	_\infty)\right)^{\theta-1} \left\langle \nabla \CalE({\gamma}), \partial_t {\gamma}\right\rangle_{L^2(\diff s_{{\gamma}})}\\
	&= -\theta \left(\CalE(\tilde{\gamma})-\CalE(\gamma
	_\infty)\right)^{\theta-1} \left\langle \nabla \CalE({\gamma})+\lambda(\gamma)\nabla\CalL(\gamma), \partial_t {\gamma}\right\rangle_{L^2(\diff s_{{\gamma}})}\\
	&=  \theta \left(\CalE(\tilde{\gamma})-\CalE(\gamma
	_\infty)\right)^{\theta-1} \norm{\nabla \CalE({\gamma})+\lambda({\gamma})\nabla\CalL({\gamma})}{L^2(\diff s_{{\gamma}})} \norm{\partial_t{\gamma}}{L^2(\diff s_{{\gamma}})},
	\end{align}
	where we used that $\langle \nabla\CalL(\gamma), \partial_t \gamma\rangle_{L^2(\diff s_{\gamma})}=0$ by \eqref{eq:length constant}. Furthermore, using the geometric transformation of the energy and the $L^2$-gradient, we find
	\begin{align}
	&\left(\CalE(\tilde{\gamma})-\CalE(\gamma
	_\infty)\right)^{\theta-1} \norm{\nabla \CalE({\gamma})+\lambda({\gamma})\nabla\CalL({\gamma})}{L^2(\diff s_{{\gamma}})}  \\
	&\quad= \left(\CalE(\tilde{\gamma}-p)-\CalE(\gamma
	_\infty)\right)^{\theta-1} \norm{\nabla \CalE(\tilde{\gamma}-p)+\lambda(\tilde{\gamma}-p)\nabla\CalL(\tilde{\gamma}-p)}{L^2(\diff s_{\tilde{\gamma}-p})}.
	\end{align}
	Thus, by the definition of $s_n$ and the {\L}ojasiewicz--Simon inequality \eqref{eq:Lojasiewicz} we have
	\begin{align}
	-\frac{\diff}{\diff t}G(t) &\geq \frac{\theta}{C_{LS}}\norm{\partial_t \gamma(t)}{L^2(\diff s_{\gamma(t)})}\quad \text{ for all }t\in [t_n, s_n).
	\end{align}
	Consequently, by \cite[Lemma 4.10 and Remark 4.11]{RuppSpener}, we find
	\begin{align}\label{eq:Loja dt G}
	-\frac{\diff}{\diff t}G(t) \geq C\norm{\partial_t \tilde{\gamma}(t)}{L^2(\diff x)}\quad \text{ for all }t\in [t_n, s_n)
	\end{align}
	for some $C=C(\CalL(\gamma_0), \CalE(f_0), \theta, C_{LS})>0$. Thus, for all $t\in [t_n, s_n)$ we obtain
	\begin{align}\label{eq:goestozero}
	\norm{\tilde{\gamma}(t)-\tilde{\gamma}(t_n)}{L^2(\diff x)}\leq \frac{1}{C} G(t_n)\to 0,
	\end{align}
	as $n\to\infty$  since $\lim_{n\to\infty}G(t_n)=0$. Now, we assume that all of the $s_n$ are finite. By continuity, \eqref{eq:goestozero} also holds for $t=s_n$ and we observe
	\begin{align}
	\norm{p(t_n)-p(s_n)}{L^2(\diff x)}^2 &=  \abs{\int_{\mathbb{S}^1}\tilde{\gamma}(t_n)\diff x-\int_{\mathbb{S}^1}\tilde{\gamma}(s_n)\diff x}^2\nonumber \\
	&\leq \int_{\S^1} \abs{\tilde{\gamma}(t_n)-\tilde{\gamma}(s_n)}^2\diff x \to 0,
	\end{align}
	using Jensen's inequality and \eqref{eq:goestozero}. Moreover, using the bounds in \eqref{eq:convbound 0} and \eqref{eq:convbound 1} we may, as at the beginning of the proof, assume that $\tilde{\gamma}(s_n)-p(s_n)\to\psi$ smoothly as $n\to\infty$ for some $\psi\in C^{\infty}(\S^1;\R^2)$. Thus, we find
	\begin{align}\label{eq:loja conv 1} \norm{\tilde{\gamma}(s_n)-p(s_n)-\gamma_\infty}{L^2(\diff x)}&\leq \norm{\tilde{\gamma}(s_n)-\tilde{\gamma}(t_n)}{L^2(\diff x)} + \norm{p(t_n)-p(s_n)}{L^2(\diff x)} \nonumber \\
	&\quad + \norm{\tilde{\gamma}(t_n)-p(t_n)-\gamma_{\infty}}{L^2(\diff x)} \to 0, \text{ as }n\to\infty.
	\end{align}
	Therefore, $\psi=f_\infty$. However, by the definition of $s_n$ and a continuity argument, we have 
	\begin{align}
		\norm{\psi-f_\infty}{W^{4,2}}=\lim_{n\to\infty}\norm{\tilde{\gamma}(s_n)-p(s_n)-\gamma_\infty}{W^{4,2}} = \sigma>0,
	\end{align}
	 a contradiction. Consequently, we have $s_{n_0}=\infty$ for some $n_0\in \N$ and therefore ${\norm{\tilde{\gamma}(t)-p(t)-\gamma_\infty}{W^{4,2}}<\sigma}$ for all $t\geq t_{n_0}$. But then \eqref{eq:Loja dt G} implies that for any $t_{n_0}\leq t \leq t'$ we have
	\begin{align}
	\norm{\tilde{\gamma}(t)-\tilde{\gamma}(t')}{L^2(\diff x)} \leq \int_{t}^{t'}\norm{\partial_t \tilde{\gamma}(\tau)}{L^2(\diff x)}\diff \tau \to 0,
	\end{align}
	as $t', t\to\infty$ by dominated convergence. Therefore, the limit $\lim_{t\to\infty}\tilde{\gamma}(t)$ has to exist in $L^2(\diff x)$ and hence equals $\gamma_\infty$. A subsequence argument shows that for any $m\in \N_0$ we have $\norm{\tilde{\gamma}(t)-\gamma_\infty}{{C}^{m}(I;\R^2)}\to 0$ as $t\to\infty$, i.e. the convergence is smooth.
\end{proof}

The bound in \Cref{thm:conv length pres} is optimal, since the stationary flow $\gamma(t)\equiv \gamma^{\ast}$ with $\gamma^{\ast}$ as in \Cref{def:figeight} solves \eqref{eq:Elastic FLow} but possesses self-intersections for all times.

\begin{thm}Let $\lambda> 0$ and let $\gamma_0\in C^{\infty}(\S^1;\R^2)$ be an embedded curve such that $$\frac{(\CalE(\gamma_0)+\lambda\CalL(\gamma_0))^2}{4\lambda}<c^*.$$ Then, the length penalized elastic flow \eqref{eq:Elastic FLow} remains embedded for all times and converges, as $t\to\infty$, after reparametrization  to a one-fold cover of a circle with radius $\frac{1}{\sqrt{2\lambda}}$.
\end{thm}
\begin{proof}
	By \cite[Theorem 3.2]{DKS} the flow exists for all times and also satisfies the bounds \eqref{eq:convbound 0} and \eqref{eq:convbound 1}. Moreover, we have the bounds
	\begin{align}\label{eq:penalized flow energy bounds}
		\CalE(\gamma(t))+\lambda\CalL(\gamma(t))\leq \CalE(\gamma_0)+\lambda\CalL(\gamma_0)
	\end{align}
	for all $t\geq 0$. An easy calculation yields 
	$\max\{ xy \mid x,y\in (0,\infty), x+\lambda y\leq M\} = \frac{M^2}{4\lambda}$ for any $M>0$, so we have $\CalE(\gamma(t))\CalL(\gamma(t))\leq \frac{(\CalE(\gamma_0)+\lambda\CalL(\gamma_0))^{2}}{4\lambda}<c^*$ by assumption.
	Thus, using \Cref{thm:main} the flow remains embedded. 
	
	As in the proof of \Cref{thm:conv length pres},  by the uniform estimates \eqref{eq:convbound 0} and \eqref{eq:convbound 1}, there exists $t_n\to\infty$ such that $\tilde{\gamma}(t_n)-p(t_n)\to\gamma_\infty$, where $\tilde{\gamma}$ is a reparametrization of $\gamma$ and $\gamma_\infty$ is a closed elastica. Here $p(t) = \int_{\S^1}\tilde{\gamma}(t) \diff x\in \R^2$. The convergence of the flow as $t\to\infty$ has been established in \cite[Theorem 1.2 and Remark 1.4]{MantegazzaPozzetta}. Alternatively, one can easily modify the arguments in the proof of \Cref{thm:conv length pres} and apply a  (unconstrained) {\L}ojasiewicz--Simon gradient inequality for the penalized elastic energy.
	 
	Since the limit is necessarily an elastica with $\CalE(\gamma_\infty)\CalL(\gamma_\infty)<c^*$, it can only be a one-fold cover of a  circle by \Cref{lem:charelas}. Denoting by $R>0$ its radius, we observe
	\begin{align}
		\partial_s^2\kappa + \frac{1}{2}\kappa^3 - \lambda\kappa = \frac{1}{2R^3}-\frac{\lambda}{R},
	\end{align}
	which equals zero if and only if $R=\frac{1}{\sqrt{2\lambda}}$.
\end{proof}

\section*{Acknowledgments}
Marius Müller was supported by the LGFG Grant (Grant no. 1705 LGFG-E).
Fabian Rupp is supported by the Deutsche Forschungsgemeinschaft (DFG,
German Research Foundation) - project no. 404870139. {Both authors would like to thank Anna Dall'Acqua for helpful discussions.}

\begin{appendices}

\section{Differential geometry in Sobolev spaces}\label{sec:AppA}

In this section, we will review some standard results from elementary differential geometry in the setting of $W^{2,2}$-curves.

\begin{thm}[Fenchel's Theorem]\label{thm:Fenchel}
	Let $\gamma\in C^2(\S^1;\R^2)$ be an immersed curve. Then $\CalK(\gamma)\geq 2\pi$ with equality if and only if $\gamma$ is  embedded and convex.
\end{thm}
\begin{proof}
	See \cite[Satz 1]{Fenchel}. An explicit characterization of the equality case can be deduced from \cite[Theorem 3]{BrickellHsiung}, for instance.
\end{proof}
\begin{lem}[Angle Function]\label{lem:liftingRegularity}
	For an immersed curve $\gamma\in W^{2,2}(\S^1;\R^2)$, there exists $\theta \in W^{1,2}((0,1);\R)$ such that
	\begin{align}\label{eq:lifting}
		\frac{\gamma^{\prime}(x)}{\abs{\gamma^{\prime}(x)}} = \begin{pmatrix}
		\cos \theta(x) \\ \sin \theta(x)
		\end{pmatrix} \text{ for all }x\in \S^1.
	\end{align}
	We call $\theta$ an \emph{angle function} for $\gamma$.
	Moreover, any two angle functions satisfying \eqref{eq:lifting} can only differ by an integer multiple of $2\pi$.  
\end{lem}

\begin{proof}
	The proof works exactly as in the case of smooth curves, see for instance \cite[Lemma 2.2.5]{Baer}. For the regularity of $\theta$, we use local representations of $\theta$. For instance, in the case $\gamma_1^{\prime}(x)>0$, one has locally
\begin{align}\label{eq:liftingArctan}
	\theta(x)\defeq \arctan\left(\frac{\gamma^{\prime}_2(x)}{\gamma^{\prime}_1(x)}\right)+ 2\pi \ell,\quad \text{for some } \ell\in \Z.
\end{align}
	Hence $\theta\in W^{1,2}(\S^1;\R^2)$ follows from \eqref{eq:liftingArctan} and the chain rule for Sobolev functions.
\end{proof}

\begin{defi}[Winding Number]\label{def:winding}
	Let $\gamma\in W^{2,2}(\S^1;\R^2)$ be an immersion with corresponding  angle function $\theta\in W^{1,2}((0,1);\R)$. We define the \emph{winding number of $\gamma$} as $T[\gamma]\defeq \frac{1}{2\pi} (\theta(1)-\theta(0))$. Note that $T[\gamma]$ does not depend on the choice of $\theta$ and is always an integer.
\end{defi}

\begin{prop}\label{thm:WindinIntegral}
	Let $\gamma\in W^{2,2}(\S^1;\R^2)$ be an immersed curve. Then 
	\begin{align}
		T[\gamma] = \frac{1}{2\pi} \int_{\S^1}\kappa \diff s.
	\end{align}
\end{prop}
\begin{proof}
	Let $\theta\in W^{1,2}((0,1);\R)$ be an angle function for $\gamma$. Differentiating \eqref{eq:lifting} and using the chain rule for Sobolev functions and the definition of the unit normal, we obtain
	\begin{align}
	\kappa \vec{n} = \vKap = \partial_s^2 \gamma = \frac{\theta^{\prime}}{\abs{\gamma^{\prime}}} \begin{pmatrix}
	-\sin \theta \\
	\cos\theta
	\end{pmatrix}
	=   \frac{\theta^{\prime}}{\abs{\gamma^{\prime}}} \begin{pmatrix}
	-\gamma^{\prime}_2 \\ \gamma^{\prime}_1
	\end{pmatrix}
	= \frac{\theta^{\prime}}{\abs{\gamma^{\prime}}} \vec{n},
	\end{align}
	consequently 
	\begin{align}\label{eq:theta_kappa}
		\kappa\abs{\gamma^{\prime}} = \theta^{\prime} \text{ almost everywhere.}
	\end{align}
	Moreover, by the fundamental theorem of calculus for $W^{1,2}$-functions, we find
	\begin{align}
		T[\gamma] = \frac{1}{2\pi}(\theta(1)-\theta(0)) = \frac{1}{2\pi} \int_{0}^{1} \theta^{\prime}\diff x = \frac{1}{2\pi}\int_{0}^1 \kappa \abs{\gamma^{\prime}}\diff x = \frac{1}{2\pi}\int_{\S^1}\kappa \diff s. &\qedhere
	\end{align}
\end{proof}

\begin{thm}[Hopf's Umlaufsatz for $W^{2,2}$-embeddings]\label{thm:HopfW22}
	Let $\gamma\in W^{2,2}(\S^1;\R^2)$ be an embedding. Then $T[\gamma]=\pm 1$.
\end{thm}
\begin{proof}
	Let $\gamma^{(n)}$ be a sequence of smooth curves with $\gamma^{(n)} \to \gamma$ in $W^{2,2}(\S^1;\R^2)$. By \Cref{thm:WindinIntegral}, we can easily see that $T[\gamma^{(n)}]\to T[\gamma]$. Since the set of embeddings is open in $C^{1}(\S^1;\R^2)$ by \Cref{lem:embeddingsopen}, we see that $\gamma^{(n)}$ is an embedding for $n\geq N$ large enough, hence $T[\gamma^{(n)}] =\pm 1$ for all $n\geq N$ by Hopf's Umlaufsatz for smooth curves. However, since the sequence $\left(T[\gamma^{(n)}]\right)_{n\in\N}$ converges, it has to be eventually constant, say $T[\gamma^{(n)}] = \tau \in \{-1, 1\}$ for all $n\geq N$. But then $T[\gamma] = \lim_{n\to\infty} T[\gamma^{(n)}] = \tau \in \{-1,1\}$.
\end{proof}

\begin{lem}\label{lem:reparaConstSpeed}
	Let $\gamma\in W^{2,2}(\S^1;\R^2)$ be an immersion. Then, there exists a constant speed reparametrization $\tilde{\gamma}$ of $\gamma$ such that $\tilde{\gamma}\in W^{2,2}(\S^1;\R^2)$.
\end{lem}
\begin{proof}
	Follows with the arguments in \cite[Proposition 2.1.13]{Baer}, using the Sobolev embedding $W^{2,2}(\S^1;\R^2)\hookrightarrow C^1(\S^1;\R^2)$ and the chain rule for Sobolev functions.
\end{proof}
%\begin{proof}
%	By the Sobolev embedding $W^{2,2}(\S^1;\R^2)\hookrightarrow C^1(\S^1;\R^2)$ we have $\inf_{\S^1}\abs{\gamma^{\prime}}>0$.
%	Let $L\colon [0,1]\to \R, L(x) \defeq \frac{1}{\CalL(\gamma)}\int_{0}^x \abs{\gamma^{\prime}(y)}\diff y$. Note that $\abs{\gamma^{\prime}}\in C(\S^1;\R)$ so $L\in C^1(\S^1;\R)$ with $L(0)=0, L(1)=1$. Since $L^{\prime}(x) = \frac{1}{\CalL(\gamma)}\abs{\gamma^{\prime}(x)}>0$, $L$ is strictly increasing on $[0,1]$ and hence there exists an inverse $\varphi\in C^{1}([0,1];\S^1)$ with $\varphi(0)=0$ and $\varphi(1)=1$. Moreover, the derivative satisfies
%	\begin{align}
%		\varphi^{\prime}(z) = \frac{\CalL(\gamma)}{\abs{\gamma^{\prime}(\varphi(z))}}, \text{ for }z\in [0,1],
%	\end{align}
%	so we see that the reparametrization $\gamma\circ\varphi \in C^{1}(\S^1;\R^2)$ is constant speed, since
%	\begin{align}
%		(\gamma\circ\varphi)^{\prime}(z) = \CalL(\gamma) \frac{\gamma^{\prime}(\varphi(z))}{\abs{\gamma^{\prime}(\varphi(z))}} \text{ for } z\in [0,1].
%	\end{align}
%	Furthermore, using $\inf_{\S^1}\abs{\gamma^{\prime}}>0$ and the usual chain rule for Sobolev functions, we find $\gamma\circ\varphi\in W^{2,2}(\S^1;\R^2)$ and
%	\begin{align}
%		(\gamma\circ\varphi)^{\prime\prime} = \CalL(\gamma) \left(\frac{\gamma^{\prime\prime}\circ \varphi\cdot \varphi^{\prime}}{\abs{\gamma^{\prime}\circ \varphi}}- \frac{\langle \gamma^{\prime\prime}\circ \varphi \cdot\varphi^{\prime}, \gamma^{\prime}\circ \varphi\rangle}{\abs{\gamma^{\prime}\circ \varphi}^3}\right)
%	\end{align}
%	in the sense of weak derivatives.
%\end{proof}

\section{Jacobi elliptic functions and Euler's elastica}\label{appendix:Jacobi}
\subsection{Elliptic functions}
We provide some elementary properties of Jacobi %Jacobian
 elliptic functions, which %are to 
can be found for example in \cite[Chapter 16]{Abramowitz}.% and in many textbooks (though these might use different definitions). 
\begin{defi}[Amplitude Function, Complete Elliptic Integrals]
Fix $m \in [0,1) $. We define the \emph{Jacobi-amplitude function} $ \am(\,\cdot\,,m) \colon \mathbb{R} \rightarrow \mathbb{R} $ with \emph{modulus} {$m$} %$p$
 to be the inverse function of 
\begin{equation}
\mathbb{R} \ni z \mapsto \int_0^z \frac{1}{\sqrt{1- m\sin^2(\theta)}} \diff\theta \in \mathbb{R}
\end{equation}
We define the \emph{complete elliptic integral of first} and \emph{second kind} by
\begin{align}
K(m) &\defeq \int_0^\frac{\pi}{2} \frac{1}{\sqrt{1- m \sin^2(\theta)}} \diff\theta,& E(m) &\defeq \int_0^\frac{\pi}{2} \sqrt{1- m \sin^2(\theta)} \diff \theta 
\end{align}
and the \emph{incomplete elliptic integral of first} and \emph{second kind} by
\begin{align}
F(x,m) &\defeq \int_0^x \frac{1}{\sqrt{1- m \sin^2(\theta)}} \diff\theta,& E(x, m) &\defeq \int_0^x \sqrt{1- m \sin^2(\theta)} \diff \theta .
\end{align}
Note that $F(\cdot,m) = \am(\cdot,m)^{-1}$. 
\end{defi}

\begin{defi}[Elliptic Functions]\label{def:B2}
For $m\in [0,1)$ the \emph{Jacobi elliptic functions} are given by
\begin{align}
\cn(\cdot,m)\colon \mathbb{R} \rightarrow \mathbb{R}, \;\; &\cn(x,m) \defeq \cos(\am(x,m)), \\ \sn(\cdot,m)\colon \mathbb{R} \rightarrow \mathbb{R}, \;\; &\sn(x,m) \defeq \sin(\am(x,m)), \\ \dn(\cdot,m)\colon \mathbb{R} \rightarrow \mathbb{R}, \;\; &\dn(x,m) \defeq \sqrt{1-m\sin^2(\am(x,m))}.
\end{align}
\end{defi}
The following proposition summarizes all relevant properties and identities for the elliptic functions. They can all be found in \cite[Chapter 16]{Abramowitz}.
%\begin{remark}
%A lot of literature on elliptic functions defines the elliptic functions using another parameter $m$ to describe the modulus. Most of the times the relation between $m$ and $p$ is $m = p^2$. 
%\end{remark}

\begin{prop}\label{prop:identities}
\leavevmode\begin{enumerate}
\item (Derivatives and Integrals of Jacobi Elliptic Functions)  For each $x \in \mathbb{R}$ and $m \in (0,1)$ we have
% \begin{align}
% \frac{\partial}{\partial x} \cn(x,p) & = - \sn(x,p) \dn(x,p), \\
% \frac{\partial}{\partial x} \sn(x,p) & =  \cn(x,p) \dn(x,p) , \\
% \frac{\partial}{\partial x} \dn(x,p) & = - p^2 \cn(x,p) \sn(x,p), \\
% \frac{\partial}{\partial x} \am(x,p) & = \dn(x,p),
% \intertext{from which one can deduce}
%  \label{eq:intdnsquared} \int_0^{K(p)} \dn^2(s,p) \diff s &= E(p).
% \end{align}
\begin{align}
\frac{\partial}{\partial x} \cn(x,m) & = - \sn(x,m) \dn(x,m), &
\frac{\partial}{\partial x} \sn(x,m) & =  \cn(x,m) \dn(x,m) , \\
\frac{\partial}{\partial x} \dn(x,m) & = - m \cn(x,m) \sn(x,m), &
\frac{\partial}{\partial x} \am(x,m) & = \dn(x,m).
\end{align}
%\begin{equation}
% \label{eq:intdnsquared} \int_0^{K(m)} \dn^2(s,m) \diff s = E(m).
%\end{equation}
\item (Derivatives of Complete Elliptic Integrals) For $m \in (0,1)$ $E,K$ are smooth and 
\begin{align}
\frac{\diff}{\diff m}E(m)  &= \frac{E(m) - K(m)}{2m}, 
& \frac{\diff}{\diff m}K(m)   = \frac{(m-1) K(m) +E(m) }{2 m(1-m)}. 
\end{align}
\item (Trigonometric Identities) For each $m \in [0,1)$ and $x \in \mathbb{R}$ the Jacobi elliptic functions satisfy
\begin{align} 
  \cn^2(x,m) + \sn^2(x,m) & = 1, &  \dn^2(x,m) + m \sn^2(x,m) &= 1 .
  \end{align}
%\item (Periodicity)
%
% Let $l \in \mathbb{Z}$ and $x \in \mathbb{R}$. 
%\begin{align} 
%  E( l \frac{\pi}{2}, m) & = l E(m), \\   
%   E(x+ l \pi, m ) & = 2 l E(m) + E(x,m) \quad \forall x \in \mathbb{R}. \\
%   F( l \frac{\pi}{2}, m) & = l F(m), \\   
%   F(x+ l \pi, m ) & = 2 l K(m) + F(x,m) \quad \forall x \in \mathbb{R}. \\
%\am(lK(m),m) & = l \frac{\pi}{2}, \\
%\am(x+ 2lK(m),m) & =  l \pi + \am(x,m)\quad \forall x \in \mathbb{R }. 
%\end{align}   
\item (Periodicity) All periods of the elliptic functions are given as follows, where $l \in \Z$ and $x \in \R$:
\begin{align} 
\am(lK(m),m) & = l \frac{\pi}{2}, & \nonumber
\cn(x+ 4 l K(m), m ) &  = \cn(x,m)  ,\\
\sn(x+ 4 l K(m), m ) &  = \sn(x,m) , &
\dn(x+ 2 l K(m), m ) &  = \dn(x,m) ,\label{eq:PeriodDN}\\
F(  \tfrac{l\pi}{2} , m )& = lK(m)   & E(  \tfrac{l\pi}{2} , m )& = lE(m) \nonumber
\end{align}
\begin{equation}
 \am(x+ 2lK(m),m)  =  l \pi + \am(x,m),
\end{equation}
\begin{equation}
F(x + l \pi, m) = F(x,m) + 2lK(m) , 
\end{equation}
\begin{equation}
E(x + l\pi , m ) = E(x,m) + 2lE(m). 
\end{equation}
\item (Asymptotics of the Complete Elliptic Integrals)
\begin{equation}
\lim_{m \rightarrow 1 } K (m) = \infty, \quad \quad \quad \quad \lim_{m \rightarrow 0 } K(m) = \frac{\pi}{2}
\end{equation}
\begin{equation}
\lim_{m \rightarrow 1 } E (m) = 1, \; \; \quad \quad \quad \quad \lim_{m \rightarrow 0 } E(m) = \frac{\pi}{2}.
\end{equation}
%where $l\in \N$
 \end{enumerate} 
\end{prop}

\subsection{Some computational lemmas involving elliptic functions}
We will also need some more advanced identities for elliptic functions.
\begin{lem} \label{lem:uniqueroot}
The map $(0,1) \ni m \mapsto 2 E(m) - K(m)$  has a unique zero $m^* \in (0,1)$. Moreover $m^* > \frac{1}{2}$.
\end{lem}
\begin{proof}
We define for $m \in (0,1)$,  $f(m) \defeq \frac{2 E(m)}{K(m)} -1$. Note that  $f$ has the same zeroes as {$m\mapsto 2E(m)-K(m)$}. %$m \mapsto \frac{2E(m)}{K(m)}$.
By Proposition \ref{prop:identities} one has 
\begin{equation}
\lim_{m \rightarrow 0 } f(m) = 1, \quad \lim_{m \rightarrow 1} f(m) = -1.
\end{equation}
and hence there has to exist  a zero of $f$. To show that it is unique, we show that $f$ is strictly decreasing, which follows immediately from the following computation
\begin{align}
\frac{\diff}{\diff m} \frac{E(m)}{K(m)} & = \frac{1}{K(m)^2} \left( \frac{E(m)- K(m)}{2m} K(m) - E(m) \frac{(m-1)K(m) + E(m)}{2 m (1-m) } \right) \nonumber
\\ &=  \frac{1}{2m  (1-m ) K(m)^2} \left( 2(1-m)  E(m)K(m) - (1-m) K(m)^2 - E(m)^2 \right) \nonumber \\ 
 & =    \frac{1}{2m (1-m) K(m)^2} \left( 2 E(m) (1-m)K(m) -(1-m)^2 K(m)^2 - E(m)^2 \right) \nonumber  \\  
 &  \quad + \frac{1}{2m (1-m) K(m)^2} \left( ((1-m)^2 - (1-m)) K(m)^2 \right) \nonumber
\\ & = \label{eq:derEK}  \frac{1}{2m (1-m) K(m)^2} \left( - (E(m) - (1-m) K(m))^2 - m(1-m) K(m)^2 \right) \\ & \leq -\frac{1}{2}.  \nonumber 
\end{align}
It remains to show that  $m^* > \frac{1}{2}$. Indeed 
\begin{align}
2 E( \tfrac{1}{2} ) - K( \tfrac{1}{2} )& = \int_0^{\frac{\pi}{2}}  \left( 2 \sqrt{1- \frac{1}{2} \sin^2( \theta)} - \frac{1}{\sqrt{1- \frac{1}{2} \sin^2(\theta)}}  \right) \; \mathrm{d}\theta
\\ & = \int_0^{\frac{\pi}{2}} \frac{\cos^2(\theta)}{\sqrt{1- \frac{1}{2}\sin^2(\theta)}} \; \mathrm{d} \theta  > 0,
 \end{align}
 which implies that $f(\tfrac{1}{2} ) > 0$ and hence by monotonicity of $f$ we find $m^* > \tfrac{1}{2}$. 
\end{proof}
\begin{lem}\label{lem:4.5}
The expression $2E(m) - K(m) + m K(m)$ is strictly positive for all $m \in (0,1)$.
\end{lem}
\begin{proof}
Let $f(m) \defeq \frac{2 E(m)}{K(m)} - 1 + m$. Note note that $f(m)$ is positive if and only if the expression in the statement is positive and $K(m) > 0$. Further note that 
\begin{equation}
\lim_{m \rightarrow 1 } f(m) = 0 .
\end{equation}
To show the claim it suffices to prove that $f' < 0$. To do so, it suffices to show that $\frac{\diff}{\diff m} \frac{E(m)}{K(m)} < - \frac{1}{2}$ for all $m \in (0,1)$.
 We have already shown in \eqref{eq:derEK} that
\begin{equation}\label{eq:diif/diff}
\frac{\diff}{\diff m} \frac{E(m)}{K(m) } = \frac{1}{2m (1-m) K(m)^2} \left( - (E(m) - (1-m) K(m))^2 - m(1-m) K(m)^2 \right) \leq - \frac{1}{2},
\end{equation}
where the last inequality was obtained by estimating the square with zero. We will show that this estimate is always with strict inequality, i.e. 
\begin{equation}\label{eq:45}
E(m) - (1-m) K(m)  \neq 0 \quad \forall m \in (0,1). 
\end{equation}
Note again first that 
\begin{equation}
\lim_{m \rightarrow 0 } ( E(m) - (1-m)K(m))  = 0.
\end{equation}
Now an easy computation yields
\begin{equation}
\frac{\diff}{\diff m} (E(m) - (1-m) K(m))   = \frac{1}{2}K(m) > 0.
\end{equation}
Therefore we obtain that 
\begin{equation}
E(m) - (1-m) K(m) > 0 \quad \forall m \in (0,1). 
\end{equation}
Hence \eqref{eq:45} is shown and thus $\frac{d}{dm} \frac{E(m)}{K(m)}< - \frac{1}{2}$ for all $m \in (0,1)$. By definition of $f$ we obtain $f' < 0$. 
\end{proof}

\begin{lem}\label{lem:e-Fzweo}
Let $m^*$ be the unique zero in Lemma \ref{lem:uniqueroot}. Then the map 
\begin{equation}
[0,2\pi) \ni x \mapsto 2E(x,m^*) - F(x,m^*)  
\end{equation}
has exactly four zeroes in $[0,2\pi)$, namely $x_1 = 0, x_2 = \frac{\pi}{2}, x_3 = \pi , x_4 = \frac{3\pi}{2}$. 
\end{lem}
\begin{proof}
Let $f: \mathbb{R} \rightarrow \mathbb{R}$ be the smooth function defined by $f(x)\defeq 2 E(x,m^*)- F(x,m^*)$. We show first that $f(0) = f( \frac{\pi}{2} ) = f( \pi) = f( \frac{3\pi}{2}) = f(2\pi)$. Indeed, by Proposition \ref{prop:identities} one has for all $l \in \mathbb{Z}$
\begin{equation}
f(l \frac{\pi}{2}) = l (2E(m^*) - K(m^*)) = 0 .
\end{equation} 
Next we show that $f'$ has four zeroes in $[0,2\pi]$. Indeed, 
\begin{equation}
f'(x) = \frac{1-2m^* \sin^2(x)}{\sqrt{1-m^* \sin^2(x)}}, 
\end{equation}
which is zero if and only if $\sin^2(x) = \frac{1}{2m^*}$ which happens exactly four times in $[0,2\pi]$ since $m^* > \frac{1}{2}$ by Lemma \ref{lem:uniqueroot}. Assume now that there exists some $x_0 \in (0,2\pi)$ apart from $0, \frac{\pi}{2}, \pi , \frac{3\pi}{2}, 2\pi$ such that $f(x_0) = 0$. We can now sort the set $\{ 0, \frac{\pi}{2}, \pi , \frac{3\pi}{2}, 2\pi ,x_0 \} =  \{ y_1, y_2 , y_3 ,y_4 ,y_5 ,y_6\}$ with $0 = y_1 < ... < y_6 = 2\pi$. Since 
\begin{equation}
f(y_1) = ... = f(y_6) = 0,
\end{equation}
by the mean value theorem for all $i \in \{ 1, ..., 5\}$ there exists some $z_i \in (y_i,y_{i+1})$ such that $f'(z_i) = 0$. This however is a contradiction to the fact that $f'$ has only $4$ zeroes. As a consequence, % conclusion,
 there exists no $x_0$ as in the assumption. The claim follows.   
\end{proof}

\begin{lem}[{{cf. \cite[Proposition B.5]{MS20}}}] \label{lem:Heumanlambda}
For all $m \in (0,1)$ one has 
\begin{equation}
E(m) \leq \frac{\pi}{2\sqrt{2}} \sqrt{2-m}.
\end{equation}
\end{lem}
\begin{proof}
The proof follows from \cite[Proposition B.5]{MS20}. Be aware that the authors there use the different notation of $m = p^2$, their definition of $E(p)$ is actually $E(p^2)$ in our notation. 
\end{proof}
\subsection{Explicit parametrization of Euler's elasticae}

%\changeMar{Here we study Euler's elasticae, i.e. smooth solutions of \eqref{eq:elasticaeq}. First we parametrize solutions explicitly and identify five types of solutions. Since in the article we are only interested in closed elasticae, we also need to classify the closed solutions. Luckily only two types of solutions may produce closed curves, as we shall see in the end of this section.} 

%\subsubsection{Parametrization}

In the sequel, we shall prove the following classification result. 

\begin{prop}\label{prop:elaclassi}
{Let $I \subset \mathbb{R}$ be an interval} and let $\gamma: I \rightarrow \mathbb{R}$ be a smooth solution of \eqref{eq:elasticaeq} for some $\lambda \in \mathbb{R}$. Then up to rescaling, reparametrization and isometries of $\mathbb{R}^2$, $\gamma$ is given by one of the following \emph{elastic prototypes}.  
\begin{enumerate}
\item  (Linear elastica) $\gamma$ is a line, $\kappa[\gamma] = 0$. 
\item (Wavelike elastica) There exists $m \in (0,1)$ such that 
\begin{equation}
 \gamma(s) = \begin{pmatrix}
 2 E(\am( s, m),m )-  s  \\-2 \sqrt{m} \cn(s,m) 
\end{pmatrix}  .
 \end{equation}
 Moreover $\kappa[\gamma] = 2 \sqrt{m} \cn(s,m).$
 \item (Borderline elastica) 
  \begin{equation}
\gamma (s ) = \begin{pmatrix}
2 \tanh(s) - s  \\ - 2 \sech(s) 
\end{pmatrix}. 
 \end{equation}
 Moreover $\kappa[\gamma] =  2 \sech(s).$ 
 \item (Orbitlike elastica) There exists $m \in (0,1)$ such that 
  \begin{equation}
 \gamma(s) = \frac{1}{m} \begin{pmatrix}
 2 E(\am(s,m),m)  + (m -2)s \\- 2\dn(s,m)
\end{pmatrix}  
\end{equation}
Moreover $\kappa[\gamma] = 2 \dn(s,m).$  
\item (Circular elastica) $\gamma$ is a circle. 
\end{enumerate}
\end{prop}

We give a proof in the rest of this section. {Suppose that $\gamma$ is parametrized by arc-length}.  We know that $\kappa$ satisfies \eqref{eq:elasticaeq}. The solutions of this equation are discussed explicitly \cite[Proposition 3.3]{Linner}. 
\begin{enumerate}
\item[(1)] (Constant curvature) $\kappa$ is constant. 
\item[(2)] (Wavelike elastica) $\kappa(s) = \pm 2\alpha \sqrt{m} \mathrm{cn}( \alpha (s-s_0), m) $ for some $m \in [0,1), \alpha > 0, s_0 \in \mathbb{R}$. In this case $\lambda = \alpha^2(2m-1).$ 
\item[(3)] (Orbitlike elastica) $\kappa(s) = \pm 2 \alpha \dn(\alpha (s-s_0), m)$ for some $m \in [0,1), \alpha > 0, s_0 \in \mathbb{R}$. In this case $\lambda = \alpha^2 (2-m).$  
\item[(4)] (Borderline elastica) $\kappa(s) = \pm 2 \alpha \mathrm{sech}(\alpha (s-s_0)) $ for some $\alpha > 0, s_0 \in \mathbb{R}$ . In this case $\lambda =  \alpha^2$. 
\end{enumerate}
We have to mention that in \cite[Proposition 3.3]{Linner} the solutions are described with a different parameter $p$ instead of $m$. In our notation there holds $m = p^2 \in (0,1)$. In contrast to our list, the classification \cite[Proposition 3.3]{Linner} also allows a wider range of $p$, namely $p \in [0,1]$. However it only distinguishes two cases --- the wavelike case and the orbitlike case. 
The remaining cases in our list arise from the limit cases $p = 0,1$ in \cite[Proposition 3.3]{Linner}. Indeed, the limit case $p=m=0$ in both the wavelike and the orbitlike case correspond to constant solutions. The limit case $p=m=1$ corresponds in both cases to the borderline elastica, as one can infer immediately from \cite[(16.15.2),(16.15.3)]{Abramowitz}.
Once expressions for the curvature are known we can find explicit parametrizations of all {these} elastica. Note that once we have parametrized the solutions for $s_0 = 0$ and `$+$' instead of `$\pm$' we can obtain all other solutions by reparametrization or reflection. Hence we consider only the cases of `$+$' and $s_0 = 0$.

From \cite[Proposition 6.1]{AnnaNetworks} it is known that each smooth solution $\gamma : I \rightarrow \mathbb{R}^2$ of \eqref{eq:elasticaeq} with some parameter $\lambda \in \mathbb{R}$ corresponds up to isometries of $\mathbb{R}^2$ and reparametrization to a solution of 
\begin{equation}\label{eq:dynsys}
\begin{cases}
\gamma_1 ''   = \sigma \gamma_2 \gamma_2', \\
\gamma_2'' = - \sigma \gamma_2 \gamma_1',  \\
\gamma_1'^2 + \gamma_2'^2 = 1 ,
\end{cases}
\end{equation}
for some $\sigma > 0$.  One can now compute that if $\gamma$ is a solution of \eqref{eq:dynsys} then $\kappa =  \gamma_2'' \gamma_1'- \gamma_1'' \gamma_2' = -\sigma \gamma_2$ and $\gamma_1' - \frac{\sigma}{2} \gamma_2^2  \equiv \mu$ for some constant $\mu \in \mathbb{R}$. The last identity can be checked by taking the derivative of the expression and using the first line of \eqref{eq:dynsys}. Following the lines of \cite[Proposition 6.1]{AnnaNetworks} we also obtain that $\lambda = - \sigma \mu$. We are also free to assume that $\gamma_1(0) =0$ as \eqref{eq:dynsys} is not affected by adding a constant to $\gamma_1$. From now on the parameters $\alpha$ and $m$ will be our main parameters. We will express $\sigma, \mu$ in terms of them and use \eqref{eq:dynsys} to obtain an explicit parametrization. %In the following we use the shorthand notation $x(s) = \gamma_1(s)$ and $y(s) = \gamma_2(s)$. 

\textbf{Case 1: Constant curvature.} This yields either lines or circles.

\textbf{Case 2: Wavelike elastica}. First we show that $\sigma = \alpha^2$ and $\mu = 1-2m$.  Note that by  point (2) of the list of possible curvatures and $\lambda= -\sigma \mu$ we have that $\mu = \frac{\alpha^2}{\sigma}(1-2m)$. In partcular, since $\kappa = - \sigma \gamma_2$ we have  
\begin{equation}\label{eq:41}
\gamma_2(s) = - \frac{2}{\sigma} \alpha \sqrt{m} \mathrm{cn}( \alpha s, m)
\end{equation} 
and since $\gamma_1' = \frac{\sigma}{2} \gamma_2^2 + \mu$ we obtain
\begin{equation}\label{eq:42}
\gamma_1'(s) =  \frac{\sigma}{2} \gamma_2(s)^2 + \frac{\alpha^2}{\sigma}(1-2m) = \frac{\alpha^2}{\sigma} ( 2m \cn^2( \alpha s, m) + 1- 2m) .
\end{equation}
Therefore using \eqref{eq:dynsys} and Proposition \ref{prop:identities} we obtain
\begin{align}
1 & =\gamma_1'(s)^2 + \gamma_2'(s)^2 \\ & =  \frac{\alpha^4}{\sigma^2}  \left( \left( 2m \cn^2(\alpha s , m ) + 1- 2m   \right)^2 + 4 m \sn^2(\alpha s , m ) \dn^2(\alpha s , m ) \right) \\  & = \frac{\alpha^4}{\sigma^2} \left( ( 1- 2m \sn^2(\alpha s , m ) )^2 + 4 m \sn^2( \alpha s , m ) ( 1- m \sn^2(\alpha s , m) ) \right) = \frac{\alpha^4}{\sigma^2}
 \end{align}
 Hence $\sigma = \alpha^2$, which implies by \eqref{eq:41} that 
\begin{equation} 
 \gamma_2(s) = - \frac{2}{\alpha}\sqrt{m} \cn(\alpha s, m ).
 \end{equation}
 We can moreover improve the formula for $\mu$ to $\mu = 1-2m$.  Moreover, using $\sigma = \alpha^2$  in \eqref{eq:42} we find 
 \begin{equation}
 \gamma_1'(s) = 2 m \cn^2(\alpha s, m) + (1-2m) = 1 - 2m \sn^2(\alpha s, m ) 
 \end{equation}
 and integrating using $\gamma_1(0)= 0$ we obtain 
 \begin{align}
 \gamma_1(s) & = s - 2m \int_0^s \sn^2(\alpha s, m) \ds = s - \frac{2}{\alpha} \int_0^{\am(\alpha s, m)} \frac{m \sin^2 \theta }{\sqrt{1- m \sin^2 \theta}} d\theta \\
& =  s - \frac{2}{\alpha} \int_0^{\am(\alpha s, m)}  \left( \frac{1}{\sqrt{1- m \sin^2 \theta }} - \sqrt{1-m \sin^2 \theta } \right) d\theta
\\ & = s - \frac{2}{\alpha} F(\am(\alpha s, m),m) +  \frac{2}{\alpha} E(\am(\alpha s , m),m ) =  \frac{1}{\alpha}  \left( 2E(\am(\alpha s , m),m ) - \alpha s \right) .
 \end{align}
 Hence for fixed $\alpha > 0$ one has $\gamma(s)= \frac{1}{\alpha} \gamma_{wave}( \alpha s) $ where $\gamma_{wave}$ is given by 
 \begin{equation}
 \gamma_{wave}(s) = \begin{pmatrix}
 2 E(\am( s, m),m )-  s  \\-2 \sqrt{m} \cn(s,m) 
\end{pmatrix}  .
 \end{equation}
 \textbf{Case 3: Orbitlike elastica.} We proceed as in the wavelike case. We first show that  $\sigma = \alpha^2 m $ and $\mu = \frac{m-2}{m}$. From point (3) in the list of curvatures and $\lambda = -\sigma \mu$ we infer that $\mu = \frac{\alpha^2 (m-2)}{\sigma}$. This leads to 
 \begin{equation}
 \gamma_2(s) =- \frac{2}{\sigma} \alpha \dn( \alpha s, m) 
 \end{equation}
 and by Proposition \ref{prop:identities}
 \begin{equation}
 \gamma_1'(s) = \frac{\sigma}{2} \gamma_2(s) ^2 + \frac{\alpha^2(m-2)}{\sigma} = \frac{\alpha^2 m}{\sigma} ( 1-  2 \sn^2(\alpha s, m) ) 
 \end{equation}
 Using \eqref{eq:dynsys} and Proposition \ref{prop:identities} we obtain 
 \begin{align}
 1 & =  \gamma_1'(s)^2 + \gamma_2'(s)^2 \\ & = \frac{\alpha^4m^2}{\sigma^2} \left( (1 - 2 \sn^2(\alpha s , m ) )^2 + 4 \sn^2(\alpha s, m) \cn^2(\alpha s , m )  \right) \\ & = \frac{\alpha^4 m^2}{\sigma^2} \left( (1 - 2 \sn^2(\alpha s, m) )^2 + 4 \sn^2(\alpha s, m ) ( 1- \sn^2(\alpha s, m) \right) = \frac{\alpha^4 m ^2}{\sigma^2}  
 \end{align} 
 Therefore we find that $\sigma = \alpha^2 m $ and from this follows that $\mu = \frac{m-2}{m}$. We infer 
 \begin{equation}
 \gamma_2(s) =- \frac{2}{m \alpha} \dn( \alpha s, m)
 \end{equation}
 and 
 \begin{equation}
 \gamma_1'(s) = 1 - 2 \sn^2( \alpha s, m).
 \end{equation}
 Integrating we obtain 
 \begin{equation}
 \gamma_1(s) = s - 2 \int_0^s \sn^2(\alpha s, m)\diff s= s \left( 1 - \frac{2}{m} \right) + \frac{2}{\alpha m} E( \am( \alpha s , m ),m ) .
 \end{equation}
 We infer that for fixed $\alpha > 0$ one has that $\gamma(s) = \frac{1}{\alpha} \gamma_{orbit}(\alpha s) $, where $\gamma_{orbit}$ is given by 
 \begin{equation}
 \gamma_{orbit} = \frac{1}{m} \begin{pmatrix}
 2 E(\am(s,m),m)  + (m -2)s \\- 2\dn(s,m)
\end{pmatrix}  .
 \end{equation}
 \textbf{Case 4: Borderline elastica} One can proceed exactly as in the first two cases and obtain that for fixed $\alpha > 0$ one has that $\gamma = \frac{1}{\alpha}\gamma_{border}(\alpha s) $ where 
 \begin{equation}
\gamma_{border} (s ) = \begin{pmatrix}
2 \tanh(s) - s  \\ - 2 \sech(s) 
\end{pmatrix}. 
 \end{equation}

\end{appendices}

\bibliography{Lib}
\bibliographystyle{abbrv}

\end{document}